\documentclass[12pt]{amsart}
\usepackage{amssymb}

\newtheorem{theorem}{Theorem}[section]
\newtheorem{lemma}{Lemma}[section]

\newcommand{\Q}{{\mathbb Q}}
\newcommand{\Z}{{\mathbb Z}}

\newcommand{\C}{{\mathbb C}}

\newcommand{\pte}{\text{PTE}}

\title[The Diophantine equation $f(x)=g(y)$]
{The Diophantine equation $f(x)=g(y)$ for polynomials with simple rational roots}

\author[L. Hajdu]{L. Hajdu}
\author[R. Tijdeman]{R. Tijdeman}

\address{L. Hajdu\newline
\indent Institute of Mathematics, University of Debrecen, \newline
\indent P. O. Box 12, H-4010 Debrecen, Hungary}
\email{hajdul@science.unideb.hu}

\address{R. Tijdeman\newline
\indent Mathematical Institute\newline
\indent Leiden University\newline
\indent Postbus 9512, 2300 RA Leiden, The Netherlands}
\email{tijdeman@math.leidenuniv.nl}

\thanks{Research of L.H. supported in part by the NKFIH grants 115479, 128088, and 130909, and the projects EFOP-3.6.1-16-2016-00022 co-financed by the European Union and the European Social Fund.}



\begin{document}

\begin{abstract}
In this paper we consider Diophantine equations of the form $f(x)=g(y)$ where $f$ has simple rational roots and $g$ has rational coefficients. We give strict conditions for the cases where the equation has infinitely many solutions in rationals with a bounded denominator. We give examples illustrating that the given conditions are necessary. It turns out that such equations with infinitely many solutions are strongly related to Prouhet-Tarry-Escott tuples. In the special, but important case when $g$ has only simple rational roots as well, we can give a simpler statement. Also we provide an application to equal products with terms belonging to blocks of consecutive integers of bounded length. The latter theorem is related to problems and results of Erd\H{o}s and Turk, and of Erd\H{o}s and Graham.
\end{abstract}

\maketitle

\section{Introduction}

Let $a_1,\dots,a_k$ be distinct rationals and $a_0\in\Q$ with $a_0\neq 0$. Put
\begin{equation}
\label{fx}
f(x)=a_0(x-a_1)\cdots(x-a_k)
\end{equation}
and let $g(y)\in\Q[y]$. In this paper we investigate for which $f,g$ equation
\begin{equation}
\label{maineq}
f(x)=g(y)
\end{equation}
has infinitely many solutions. Moreover, we study for which $f,g$ this holds if $g$ is of the form 
\begin{equation}
\label{gy}
g(y)=b_0(y-b_1)\cdots(y-b_\ell),
\end{equation}
where $b_1,\dots,b_\ell$ are distinct elements of $\Q$ and $b_0\in\Q$ with $b_0\neq 0$. We say that an equation $f(x)=g(y)$ has infinitely many rational solutions with a bounded denominator if there exists a positive integer $\Delta$ such that $f(x)=g(y)$ has infinitely many solutions $(x,y)\in \Q^2$ with $(\Delta x, \Delta y) \in \Z^2$. Our focus is the question for which $f,g$ equation \eqref{maineq} has infinitely many solutions $(x,y)\in \Q^2$ with a bounded denominator.

Using results of Bilu and Tichy \cite{bt} and of Davenport, Lewis and Schinzel \cite{dls}, both based on a theorem of Siegel \cite{sie}, we prove the following theorem.

\begin{theorem}
\label{caseA}
Let $f(x) \in \Q[x]$ have only simple rational roots and let $g(x) \in \mathbb{Q}[x]$. 
Suppose the equation $f(x)=g(y)$ has infinitely many solutions $(x,y) \in\Q^2$ with a bounded denominator. 

Then there exist $m \in \{1,2,3,4,6\}$, $n,s \in \Z_{>0}$ or $n \in \{1,2\}, m,s \in \Z_{>0}$ such that $\deg(f)=ms$, $\deg(g)=ns$.

If also $g$ has only simple rational roots and $\deg(f) \leq \deg(g)$, then there exist $m \in \{1,2\}$, $n,s \in \Z_{>0}$ such that $\deg(f)=ms$, $\deg(g)=ns$.
\end{theorem}

The first statement will be proved in Section 7. After the proof we shall argue that if $m \in \{1,2,3,4,6\}$ for every such $m,n,s$ a pair of polynomials $(f,g)$ can be constructed with $f$ having only simple rational roots such that $f(x)=g(y)$ has infinitely many integral solutions $(x,y)$. For the remaining cases, see Section 11. 

The second statement will be proved in Section 9. Observe that it follows that $\deg(f) \mid 2\deg(g)$. 
\vskip.2cm
As illustration of Theorem \ref{caseA} we present some nontrivial examples. Later more examples will follow.

\vskip.2cm

\noindent
{\bf Example 1.1.} An example of the second statement where $\deg(f)$ does not divide $\deg(g)$.
Let
$$
f(x)=(x-6)(x+6), \ \ \ g(y)= (y-1)(y-4)(y-9).
$$
Then $f(x)=g(y)$ has solution
$$
(x,y)=(X(X^2-7), X^2)\ \ \ \text{for every}\ X\in\Z.
$$

\vskip.2cm

\noindent
{\bf Example 1.2.} An example of the second statement where $\deg(f)$ divides $\deg(g)$. Let
$$
f(x)=(x-7)(x-1)(x+1)(x+7),~~ g(y)=4(y-5)(y-1)(y+1)(y+5).
$$
Consider the Pell equation $x^2 = 2y^2-1$. It has solutions $(X_i,Y_i)_{i=1}^{\infty}$ given by $(X_1,Y_1) = (1,1)$, $(X_2,Y_2) =(7,5)$ and
$$
X_{i+1} = 6X_i-X_{i-1}, \ \ \ Y_{i+1}=6Y_i-Y_{i-1}\ \ \ (i=2,3, \dots).
$$
Then $f(x)=g(y)$ has as solution
$$
(x,y) = (X_i,Y_i)\ \ \ \text{for every}\ i\in\Z_{>0}.
$$

\vskip.2cm

\noindent
{\bf Example 1.3.} An example of the first statement for $m=3$, $n=4$ and $s=1$. Let
$$
f(x) = (x+286)(x+13)(x-299),~~ g(y) = y^4 - 8788y^2 +8541936.
$$
For every $X\in \Z$ there is a solution
$$
(x,y) = (X^4-52X^2+338, X^3-39X).
$$



\vskip.2cm

In Section \ref{ho} we give a historical overview of the literature on equations $f(x)=g(y)$ where $f$ has only simple rational roots. In Section \ref{secbt} we present the Bilu-Tichy decomposition which is fundamental for our treatment. Bilu and Tichy distinguish five kinds of standard pairs. We exclude the fifth kind and rephrase Theorem \ref{caseA} as Lemma \ref{btcons}. In Section \ref{secpte} we present Prouhet-Tarry-Escott (PTE-)sets, an extension of ideal PTE-pairs. In Section \ref{sec12} we consider standard pairs of the first and second kind where $g$ need not satisfy \eqref{gy}. In the next section we assume that $g$ satisfies \eqref{gy} too. Section \ref{sec23} deals with standard pairs of the third and fourth kind where $g$ need not satisfy \eqref{gy}. Section \ref{sec34} restricts the cases with standard pairs of the third and fourth kind if $g$ has only simple rational roots, too. In Section \ref{secnew} we give a more precise statement than Theorem \ref{caseA} under \eqref{gy} which completes the proof of Theorem \ref{caseA}. We give an application of our results to equal products with terms belonging to blocks of consecutive integers of bounded lengths in Section \ref{seceb}. We finish the paper with some conclusions and open problems.

\section{Historical overview}
\label{ho}

There are numerous publications on the title equation where $f$ has only simple rational roots. In most of them the roots of  $f$ and $g$ form almost arithmetic progressions. This overview is far from complete and the results in the mentioned papers are mostly more general than cited.

\subsection{The roots of $f$ form an arithmetic progression and $g$ is almost a perfect power.}
First we consider the case that the roots of $f$ form an arithmetic progression and $g$ is almost a perfect power, more precisely:
\begin{equation}
\label{euler}
x(x+d)\cdots(x+(k-1)d)=b_0y^{\ell}+b_{\ell}
\end{equation}
where $b_0$, $b_{\ell}$, $d$, $k$ and $\ell$ are integers with $k>1$, $\ell > 1$, $k\ell >4$, $b_0 \neq 0$ $\ell$-th power free, the greatest prime factor of $b_0$ is at most $k$ and solutions $(x,y) \in \Z^2$ satisfy $\gcd(x,d)=1$, $y>1$. (If $k= \ell = 2$, then we may have a Pell equation which has infinitely many solutions.) If $b_{\ell}=0$, then there are only finitely many solutions according to a theorem of Siegel \cite{sie26} if $\ell>2$ and by a result of Schinzel \cite{sch67}, Corollary 7 if $\ell = 2$. 

Let $d=1$. In 1975 Erd\H{o}s and Selfridge \cite{es} proved that equation \eqref{euler} has no solutions when $b_0=1$, $b_{\ell}=0$. Erd\H{o}s \cite{euj} for $k \geq 4$ and Gy\H{o}ry \cite{gyuj} for $k=2,3$ showed that the equation ${x+k-1 \choose k} = y^{\ell}$, which agrees with the case $b_0=k!$, $b_{\ell}=0$ in \eqref{euler}, has only the solution ${50 \choose 3} = 140^2$. Saradha \cite{sar97} for $k\geq 4$ and Gy\H{o}ry \cite{gy98} for $k=2,3$ proved that equation \eqref{euler} with $b_0>1$, $b_{\ell}=0$ has no solution provided that the greatest prime factor of the left-hand side is larger than $k$. Bilu, Kulkarny and Sury \cite{bks} proved that equation \eqref{euler} has only finitely many solutions $(k, \ell, m,n)$ if $b_{\ell}$ is not a perfect power and that all solutions can be explicitly determined. For more results with $d=1$ see \cite{coh}, \cite{gp}, \cite{yua}.

Next let $d>1$, $b_0=1$, $b_{\ell}=0$. A famous result due to Euler is that the product of four distinct positive integers in arithmetic progression cannot be a square. Gy\H{o}ry, Hajdu and Saradha \cite{gyhs} generalized this by proving that the product of four or five consecutive terms of an arithmetic progression with $d$, $x$ coprime cannot be a perfect power. This result has been extended to at most $11$ terms by Bennett, Bruin, Gy\H{o}ry and Hajdu \cite{bbgh} and to at most $34$ terms by Gy\H{o}ry, Hajdu and Pint\'er \cite{gyhp}. Another generalization of Euler's result in the square case by Hirata, Laishram, Shorey and Tijdeman \cite{hlst} extending \cite{bbgh}, \cite{gyhp}, \cite{obl} is that equation \eqref{euler} with $4 \leq k \leq 109$, $\ell = 2$ has no solutions. For a similar result for $\ell = 3$ see Hajdu, Tengely and Tijdeman \cite{htt}, and for $\ell = 5$ see Hajdu and Kov\'acs \cite{hk11}. Bennett \cite{ben18} obtained the following strong finiteness result: There exist at most finitely many integer tuples $d,k,\ell,x,y$, with $4 \leq k\leq 15177$ for which equation \eqref{euler} is satisfied. Bennett and Siksek \cite{bs20} proved for some effectively computable $k_0$ that for fixed $k>k_0$ there are only finitely many integers $d,\ell,x,y$ satisfying equation \eqref{euler}. For some other papers in this case see \cite{fls12}, \cite{sar97}, \cite{sar98}, \cite{sh96}. 

Case $d>1$, $b_0>1$. Saradha and Shorey \cite{ss05} proved that for $d$ at most some explicitly given $d_0=d_0(\ell)$ and $b_{\ell}=0$ equation \eqref{euler} has no solutions. It follows from Yuan \cite{yua} that if $k \geq 8$ then all solutions of \eqref{euler} satisfy $\max(x, y, \ell) <C$ where $C$ is an effectively computable constant depending only on $k, b_0, b_{\ell}$. For other results with $b_{\ell}=0$ see \cite{fls12}, \cite{ls07}, \cite{ms03}, \cite{ss031}, \cite{ss05}. For a more complete survey of papers on equation \eqref{euler} see \cite{sh06}.

\subsection{The roots of $f$ form almost an arithmetic progression and $g$ is almost a perfect power.}
First we turn to the case that the roots of $f$ form an arithmetic progression with some terms missing, more precisely, to the equation
\begin{equation}
\label{shorey}
(x+d_1d)\cdots(x+d_{k}d)=b_0y^{\ell} + b_{\ell}
\end{equation} 
where $0 \leq d_1 < d_2< \dots < d_{k} < K$, $d, b_0, b_{\ell}$ and $\ell$ are integers with $k>2$, $\ell > 1$, $b_0$ is $\ell$-th power free, the greatest prime factor of $b_0$ is at most $k$ and solutions $(x,y) \in \Z^2$ satisfy $\gcd(x,d)=1$.

Several papers deal with the case $K-k=1$. Saradha and Shorey \cite{ss01}, Hanrot, Saradha and Shorey \cite{hss01} and Bennett \cite{ben04} together proved that for $d=K-k=b_0=1$, $b_\ell=0$ the only solutions of \eqref{shorey} are given by $4!/3=2^3$, $6!/5 = 12^2$, $10!/7 = 720^2$. For other papers with $K-k=1$, $b_\ell=0$ see \cite{dls18a}, \cite{ss03}, \cite{ss031}, \cite{ss07}, \cite{ss15}. Hajdu and Papp \cite{hp20} proved that equation \eqref{shorey} with $K-k=1$, $K \geq 8$ has only finitely many solutions $x,y,\ell$.

Mukhopadhyay and Shorey \cite{ms03} for $\ell=2$ and Saradha and Shorey \cite{ss08} for $\ell \geq 3$ determined all solutions of equation \eqref{shorey} with $K-k=2$, $k \geq 4$, $\ell \geq 3$. For papers with $K-k \geq 2, b_{\ell}=0$ see \cite{bs93} and \cite{dls18}. Hajdu, Papp and Tijdeman \cite{hpt} provided effective upper bounds for $\max(|x|,|y|,\ell)$ in \eqref{shorey} under the assumption that $K-k < cK^{2/3}$ for some explicit $c>0$. 

In the case when instead of omitting terms from an arithmetic progression we have an extra term we recall some results concerning so-called figurate numbers. The $x$-th figurate number with integer parameters $k,m$ is defined as $f_{k,m}(x)=\frac{x(x+1)\dots (x+k-2)((m-2)x+k+2-m)}{k!}$. Note that these numbers are generalizations of factorials, and have been intensively studied by many authors. Here, in relation with \eqref{shorey} we only mention that Hajdu and Varga \cite{hv} proved that the equation
\begin{equation}
\label{fig}
f_{k,m}(x)=b_0y^{\ell} + b_{\ell}
\end{equation}
with $k\geq 3$, $m\geq 4$, $b_0,b_\ell\in\Q$, $b_0\neq 0$ has only finitely many solutions $(x,y,\ell) \in \Z^3$ with $\ell\geq 2, |y| > 1$, unless $(k,\ell)=(3,2)$ or $(k,m,\ell) =(4,4,2),(4,6,2),(4,4,4)$. For many related results and history, see \cite{hv}.

\subsection{Both $f$ and $g$ have simple rational roots almost in arithmetic progressions.} 

In the literature many papers deal with special cases of the equation
\begin{equation}
\label{bst}
a_0x(x+d_1)\cdots (x+(k-1)d_1) = b_0y(y+d_2)\cdots (y+(\ell-1) d_2)
\end{equation}
where $k,\ell, a_0, b_0$ are integers with $1 <k \leq  \ell$, $a_0b_0 \neq 0$, and $d_1,d_2$ are positive integers with $d_1\neq d_2$ if $k=\ell$.

First the case $a_0=b_0=d_1=d_2=1$ attracted attention. In 1963 Mordell \cite{mor63} proved that for $(k,\ell) = (2,3)$ the only positive integer solutions are given by $(x,y) = (2,1)$ and $(14,5)$. In 1972 Boyd and Kisilevsky \cite{bk} proved that $(x,y)=(2,1), (4,2), (55,19)$ are the only positive integer solutions if $(k,\ell) = (3,4)$, while Hajdu and Pint\'er \cite{hpi} showed that the only positive integer solution for $(k,\ell) = (4,6)$ is $(7,2)$. Several results are covered by the theorem of Saradha and Shorey \cite{ss90} that the only solution with $\ell = 2k$ is given by $(k, \ell,x,y) = (3,6,8,1)$. They, together with Mignotte (see \cite{miso}) determined all solutions in case $\ell / k \in \{3,4,5,6\}$.

Saradha, Shorey and Tijdeman \cite{sst95} studied the cases $a_0=b_0=1$, $d_1=1$, $d_2>1$, $\ell / k$ is integral. All cases with $a_0=b_0=1$ were covered by Beukers, Shorey and Tijdeman \cite{bst}. They proved that equation \eqref{bst} admits only finitely many positive integral solutions $x,y$ except for the infinite class of solutions $x=y^2 + 3d_2y$ when $k=2$, $\ell = 4$ and $d_1=2d_2^2$. By a similar reasoning the restriction $a_0=b_0=1$ can be replaced by $\ell > 2$.

By taking $a_0 = \ell !$, $b_0 = k!$, $d_1=d_2$, $m=x+k-1$, $n=y+\ell -1$ in \eqref{bst} the question becomes which binomial coefficients are equal,
\begin{equation} \label{bin}
{m \choose k} = {n \choose \ell}.
\end{equation}
Without loss of generality we assume $1<k < \ell$, $k \leq m/2$, $\ell \leq n/2$. In 1966 Avanesov \cite{ava} provided all solutions to equation \eqref{bin} if $(k,\ell)=(2,3)$, and in 1963 Mordell \cite{mor63} did so for the case $(k,\ell) = (3,4)$. Stroeker and de Weger \cite{sdw} dealt with $(k,\ell) = (2,6),(2,8),(3,6),(4,6),(4,8)$ and Bugeaud, Mignotte, Siksek, Stoll and Tengely \cite{bmsst} solved the case $(k, \ell) = (2,5)$. Several pairs $(k, \ell)$ were treated by other authors. 
Gallegos-Ruiz, Katsipis, Tengely and Ulas \cite{gktu} completely solved the equations
$$ {m \choose k} = {n \choose \ell} + d, \ \ \  -3 \leq d \leq 3$$ for pairs $(k,\ell) = (2,3), (2,4), (2,6), (2,8), (3,4), (3,6), (4,6), (4,8).$ 
\noindent It follows from their work that for every integer $r$ the equation
$$(x+r)(x-r-1) = 2 {y \choose 5}$$ has two surprising large solutions.\footnote{Tengely tells us that the correct conjecture on page 434 of their paper reads that the only solutions in positive integers of the equation ${y \choose 2} = {x \choose 5} + 66$ are $(x,y) = (1,12), (2,12), (3,12), (4,12), (11,33), (28,444), (7935, 723632383), (7939, 724544908)$. Note that the equation $(y+11)(y-12) = 2 { x \choose 5}$ (i.e. $r=11$) has the same solutions.}
Surveys on (almost) equal binomial coefficients are Blokhuis, Brouwer, de Weger \cite{bbw}\footnote{In their list on page 2 the sporadic solution $n=78, k=2, m=15, \ell = 5$ is missing.} and Gallegos-Ruiz et al. \cite{gktu}.

A generalization concerns the equation
$$
f_{k,m}(x)=f_{\ell,n}(y)
$$
of figurate numbers (defined in the previous subsection). We only mention that Hajdu, Pint\'er, Tengely and Varga \cite{hptv} obtained various finiteness results concerning this equation. For history and further related results see \cite{hptv}.

\subsection{The roots of $f$ are simple and rational and $g(y)\in \Q[y]$.}
Consider the equation
\begin{equation}
\label{hajdu}
f(x):=(x+d_1d)\cdots(x+d_kd) = g(y)
\end{equation} 
in integers $x,y$ where  $d,k,K,d_1, d_2, \dots, d_k$ are integers with $0 \leq d_1 < d_2< \dots <d_k<K$ and $k>2$, $g(y) \in \Q[y]$ of degree $\ell \geq 2$. Kulkarni and Sury \cite{ks} proved that if $d=1$, $k=K$, $\ell>2$ and \eqref{hajdu} has infinitely many solutions, then either $g=f(G)$ for some $G(y) \in \Q[y]$, or $k$ is even and $g=\varphi(G)$ where
$$
\varphi(x) = \left(x-\left(\frac 12\right)^2\right) \left(x-\left(\frac 32\right)^2\right) \cdots \left(x-\left(\frac{k-1}{2}\right)^2\right)
$$
and $G(y) \in \Q[y]$ is a polynomial whose squarefree part has at most two roots, or $k=4$ and $g(y) = bv(y)^2 + 9/16$ where $b\in \Q$, $b\neq 0$ and $v$ is a linear polynomial with rational coefficients. They showed in particular that there are only finitely many such solutions $k,x,y$ if $g$ is irreducible. Hajdu, Papp and Tijdeman \cite{hpt} proved the finiteness of the number of solutions of \eqref{hajdu} under the assumption that $K-k \leq cK^{2/3}$ with $c$ an explicit constant, provided that $g$ does not belong to two explicitly given classes in which there can be infinitely many solutions. The latter two papers are based on a theorem of Bilu and Tichy \cite{bt}, which will also play an important role in our present study and is formulated in the next section. For other papers related to \eqref{hajdu} see \cite{bs93}, \cite{bbkpt}, \cite{rak03}, \cite{rak04}, \cite{stt}. For finiteness results when $f(x)$ in \eqref{hajdu} is replaced by $f_{k,m}(x)$ (related to figurate numbers), see \cite{hptv}, \cite{hv}, and the references there.

\subsection{Power values and equal values of products with terms coming from an interval.}
Finally, we recall some papers and results from the literature concerning products with terms coming from blocks of consecutive integers. These are related to our results in Section \ref{seceb}.

First we mention a result of Erd\H{o}s and Turk \cite{ert}. They (among others) studied the existence of terms from `short' intervals $I$ having a power product, and also the existence of two distinct sets of integers in $I$ with equal product. Roughly speaking, they proved that these properties never hold for `very short' intervals; that they hold in infinitely many cases, but also fail in infinitely many cases for `medium sized' intervals; and that they always hold if the size of $I$ is `large enough'. They gave precise formulas for the sizes in their paper \cite{ert}.

Another problem of somewhat similar flavor is due to Erd\H{o}s and Graham \cite{erg}, who asked when the product of two or more disjoint blocks of consecutive integers can be a power infinitely often. Ulas \cite{u} was the first to provide examples yielding a positive answer: he exhibited families of blocks of precisely four integers whose product gives perfect squares. Bauer and Bennett \cite{babe} described the `minimal examples' yielding perfect square products. For related results, see  \cite{ska}, \cite{tv} and the references there.

\section{The Bilu-Tichy theorem}
\label{secbt}

We say that a polynomial $f$ as in \eqref{fx} is symmetric, if there exists an $a\in\Q$ such that the set $\{a_1,\dots,a_k\}$ is symmetric around $a$.

We call polynomials $f, \tilde{f}\in\Q[x]$ similar if there exist $a,b\in\Q$, $a\neq 0$ such that $f(x) = \tilde{f}(ax+b)$. Notation $f \simeq \tilde{f}$. Obviously this induces an equivalence relation on $\Q[x]$. Observe that if $f$ has only simple rational roots, then $\tilde{f}$ has only simple rational roots too. In every equivalence class there are polynomials with sum of roots equal to $0$. Moreover, if the roots of $f$ are all rational, then there exists a similar polynomial $\tilde{f}(x) \in \Z[x]$ of which the roots are integers with sum $0$. If the polynomial equation $f(x)=g(y)$ has infinitely many solutions $(x,y) \in \Q^2$ with a bounded denominator and $f \simeq \tilde{f}$, $g \simeq \tilde{g}$, then the equation $\tilde{f}(x) = \tilde{g}(y)$ has also infinitely many solutions $(x,y) \in \mathbb{Q}^2$ with a bounded denominator. We call equations $f(x)=g(y)$ and $\tilde{f}(x) = \tilde{g}(y)$ with $f \simeq \tilde{f}, g \simeq \tilde{g}$ similar equations.

We call $f(x)\in \Q[x]$ decomposable over $\Q$ if there exist $G(x), H(x) \in \Q[x]$ with $\deg(G)>1$, $\deg(H) >1$ such that $f = G(H)$, and otherwise indecomposable. Since $\deg(f) = \deg(G)\cdot \deg(H)$, $f$ is indecomposable if $\deg(f)$ is prime.

Let $\delta$ be a non-zero rational number and $\mu$ be a positive integer. Then the $\mu$-th Dickson polynomial is defined by
$$
D_\mu(x,\delta):=\sum_{i=0}^{\lfloor\mu/2\rfloor}d_{\mu,i}x^{\mu-2i}\ \ \ \text{where}\ d_{\mu,i}=\frac{\mu}{\mu-i} \binom{\mu-i}{i}(-\delta)^i.
$$
For properties of Dickson polynomials see e.g. \cite{lmt}.

In this section we prove a variant of Theorem \ref{caseA}. In the proof the following result of Bilu and Tichy \cite{bt} on equation \eqref{maineq} is crucial. Here the polynomials $F,G\in\Q[x]$ form a standard pair over $\Q$ if either $(F(x),G(x))$ or $(G(x),F(x))$ appears in Table 1.

\begin{theorem}[Bilu, Tichy \cite{bt}, Theorem 1.1]
\label{lembt}
Let $f(x),g(x)\in\Q[x]$ be non-constant polynomials. Then the following two statements are equivalent.
\begin{enumerate}
	\item[(I)] The equation $f(x)=g(y)$ has infinitely many rational solutions $x,y$ with a bounded denominator.
	\item[(II)] We have $f=\varphi( F(\kappa))$ and $g=\varphi( G(\lambda))$, where $\kappa(x),\lambda(x)\in\Q[x]$ are linear polynomials, $\varphi(x)\in\Q[x]$, and $F(x),G(x)$ form a standard pair over $\Q$ such that the equation $F(x)=G(y)$ has infinitely many rational solutions with a bounded denominator. 
\end{enumerate}
\end{theorem}

\begin{table}
\begin{center}
\begin{tabular}{|c|c|c|}
\hline
Kind & Standard pair (unordered)& Parameter restrictions \\
\hline
First & $(x^q ,\alpha x^p v(x)^q)$ & $0\leq p<q$, $(p,q)=1$,\\
&& $p+\deg(v)>0$ \\
\hline
Second & $(x^2 , (\alpha x^2+\beta)v(x)^2)$ & - \\
\hline
Third & $(D_\mu(x,\alpha^\nu),D_\nu(x,\alpha^\mu))$ & $\gcd(\mu,\nu)=1$ \\
\hline
Fourth & $(\alpha^{-\mu/2}D_\mu(x,\alpha),-\beta^{-\nu/2}D_\nu(x,\beta)$) & $\gcd(\mu,\nu)=2$ \\
\hline
Fifth & $((\alpha x^2-1)^3,3x^4-4x^3)$ & - \\
\hline
\end{tabular}
\end{center}
\caption{{\small Standard pairs. Here $\alpha,\beta$ are non-zero rational numbers, $\mu,\nu,q$ are positive integers, $p$ is a non-negative integer, $v(x)\in \Q[x]$ is a non-zero, but possibly constant polynomial.}}
\end{table}

Observe that $F(\kappa) \simeq F$ and $G(\lambda) \simeq G$. The Bilu-Tichy theorem implies that if (I) holds then the equation $F(\kappa(x)) = G((\lambda(y))$ has infinitely many rational solutions with a bounded denominator. The converse is obvious. 

In Theorem \ref{caseA} one may read $m=\deg(F), n=\deg(G), s=\deg(\varphi)$.
\vskip.2cm
An interesting result in connection with Theorem \ref{lembt} is due to Avanzi and Zannier \cite{az}. Namely, Theorem 1 of \cite{az} implies that if the equation $f(x)=g(y)$ with $f(x),g(x)\in \Q[x]$, $\gcd(k,\ell)=1$ and $k,\ell >6$ has infinitely many rational solutions, then infinitely many of them have a bounded denominator. (Cf. Bilu's MathSciNet review MR1845348 of that paper.)

\vskip.2cm

We start with investigating when the equation 
\begin{equation}
\label{FGeq}
F(x)=G(y)
\end{equation}
for standard pairs $(F,G)$ has infinitely many solutions $(x,y)$ with a bounded denominator in our settings. Lemma \ref{btcons} shows that condition \eqref{fx} restricts the possibilities.

\begin{lemma}
\label{btcons}
Suppose $f$ is of the form \eqref{fx} and equation \eqref{maineq} has infinitely many rational solutions with a bounded denominator. Let $(F,G)$ be a corresponding standard pair. Then one of the following cases holds:
\begin{itemize}
\item[1)] $(F,G)$ is of the first or second kind, $\min(\deg(F),\deg(G))\leq 2$,
\item[2)] $(F,G)$ is of the third or fourth kind.
\end{itemize}
\end{lemma}

\begin{proof} 
Without loss of generality we may assume $f = \varphi(F)$, $g = \varphi(G)$. Since $f$ has only simple rational roots, $f' = \varphi'(F)F'$ has only simple real roots. Hence $F'$ has only simple real roots. If $(F,G)$ is of the fifth kind, then $F'$ has a multiple root and so the fifth kind is excluded. Therefore, if we are not in case 2), we have a pair $(F,G)$ of the first or second kind. By 1) we may assume that $\deg(F)\geq 3$ and $\deg(G) \geq 3$. Then $(F,G)$ is not of the second kind. If $(F,G)$ is of the first kind, then $q\geq 3$ and if $\deg(v)=0$ then $p\geq 3$. However, then $F'$ has a multiple root, which is not the case. 
\end{proof}

\vskip.2cm 

\noindent
{\bf Remark 3.1.} It follows that if $(F,G)$ is a standard pair of the first or second kind, then $\deg(f)\mid 2\deg(g)$ or $\deg(g)\mid 2\deg(f)$.

\vskip.2cm 

\noindent {\bf Remark 3.2.} In Examples 1.1, 1.2, 1.3 we may take
$$
F(x)= x^2,~~ G(y) = y(y-7)^2,~~ \varphi(x) = x-36,
$$
$$
F(x) = x^2,~~ G(y) = 2y^2-1,~~ \varphi(x) = (x-1)(x-49),
$$
$$
F(x) = D_3(x, 13^4),~~ G(y) = D_4(y, 13^3),~~ \varphi(x) = x-1111682,
$$
respectively.

\section{pte$_k$ polynomial sets}
\label{secpte}

Let $f(x) \in \Q[x]$ with only simple rational zeros be decomposable over $\Q$ as $\varphi(F(x))$. Let
$$
\varphi(x) = p_0(x-p_1)(x-p_2) \cdots(x-p_s)
$$
with $s>0$, $p_0\in\Q$ $(p_0\neq 0)$ and $p_i\in\C$ $(i=1,\dots,s)$. Then
$$
f(x) = p_0(F(x)-p_1)(F(x)-p_2) \cdots (F(x)-p_s).
$$
From this, we see that $p_i\in\Q$ $(i=1,\dots,s)$, and that these numbers are distinct. Further, writing $F_i(x) = F(x)-p_i$ for $i=1,2, \dots, s$ we obtain that $F_1(x), F_2(x), \dots, F_s(x) \in \Q[x]$ are such that $F_i(x)/F_j(x) \notin \Q$, $F_i(x) - F_j(x) \in \Q$ for $1 \leq i < j \leq s$ and, moreover, $F_i(x)$ has only simple rational roots for $1 \leq i \leq s$. These polynomials have the same degree, $m$ say. It follows that there are rationals $r_1, r_2,  \dots, r_m$ independent of $i$ such that $F_i(x) = r_mx^m + r_{m-1}x^{m-1} + \ldots +r_1x +f_i$ for all $i$ with $f_1, f_2, \dots, f_s \in \Q$ distinct. We call $f$ a $\pte_m$-polynomial, $\{F_1, F_2, \dots, F_s\}$ a $\pte_m$ set and $F$ a $\pte_s$ component of $f$. Of course $\deg(f)=ms$. Note that every polynomial with only rational roots is a $\pte_1$ component of itself. On the other hand, $\pte_m$ polynomials of degree $>m>1$ are decomposable.

If $\{F_1, F_2, \dots, F_s\}$ is a $\pte_{m}$ set, then the first $m-1$ symmetric polynomials of the roots of $F_i(x)$ are independent of $i$. By the formulas of Newton-Girard we obtain that the sum of the $j$-th powers of the roots of $F_i$ are independent of $i$ for $j=1,2, \dots, m-1$. In case all the roots are rational, the union of the sets of roots is called an ideal Prouhet-Tarry-Escott set. Ideal Prouhet-Tarry-Escott pairs (i.e. corresponding to the case $s=2$) are known for $2 \leq m \leq 10$ and for $m=12$. For general information on the PTE-problem we refer to \cite{rava}. In this section we shall show that for $m \in \{3,4,6\}$ we can construct arbitrarily large integral $\pte_m$ sets, that is $s$ sets of $m$ integers having the same sums of $j$-th powers for $1\leq j\leq m-1$, with $s$ arbitrary. In our construction $r_2=0$ for $m=3$ and for $m=4,6$ we have $p_{\kappa + m/2} = - p_\kappa$ for $\kappa=1,2, \dots, m/2$ and therefore $r_i = 0$ for $i$ odd. Note that $\pte_m$ sets turn into $\pte_m$ sets under linear transformations. Therefore, if there exists a $\pte_m$ set of $m$ rationals, there exists a similar $\pte_m$ set of $m$ integers. (In the literature, see \cite{rava}, $\pte_m$ sets of integers are considered. However, in view of the last remark, one can work with rational $\pte_m$ sets as well.)

\vskip.3cm

\noindent
{\bf Case $m = 4$.}
Choosing $F$ as an even polynomial, it suffices to prove that for any $s$ there are $s$ monic polynomials $F_i(x)$ of degree $4$ with distinct integer roots $\alpha_{i1}$, $\alpha_{i2}$, $-\alpha_{i1}$, $-\alpha_{i2}$ such that $2\alpha_{i1}^2 + 2\alpha_{i2}^2$ has the same value, independent of $i$. Let $M$ be the product of $\rho$ distinct primes of the form $\equiv 1\pmod{4}$. Then the number of representations of $M$ as    
$\alpha_1^2 + \alpha_2^2$ with $\alpha_1, \alpha_2 \in \mathbb{Z}$, $\alpha_1 > \alpha_2 > 0$, $\gcd(\alpha_1, \alpha_2) = 1$ equals $2^{\rho-1}$ according to Theorem 7.5 of \cite{lev}. Thus for every $\rho$ we can construct $2^{\rho - 1}$ distinct primitive polynomials $F_i$ of degree $4$ differing only in their constant terms.

\vskip.2cm

\noindent
{\bf Example 4.1.} We have
$$
5\cdot 13\cdot 17 = 1105 = x^2+y^2 ~~{\rm for} ~~ (x,y)= (33,4),(32,9), (31,12), (24,23).
$$
Hence the polynomial $P(x) = x^4 - 1105x^2$ has simple rational roots when $17424$, $82944$, $138384$ or $304704$ is added, since the corresponding polynomials equal
{\small $$(x^2-33^2)(x^2-4^2), (x^2-32^2)(x^2-9^2), (x^2-31^2)(x^2-12^2), (x^2-24^2)(x^2-23^2),$$}
respectively.

\vskip.3cm

\noindent
{\bf Case $m=6$.} It suffices to prove that for any $s$ there are $s$ monic polynomials $F_i(x)$ of degree $6$ with distinct integer roots $\pm \alpha_{i1}$, $\pm \alpha_{i2}$, $\pm \alpha_{i3}$ such that both $\sum_{\kappa = 1}^3 \alpha_{i\kappa}^2$ and $\sum_{\kappa = 1}^3 \alpha_{i\kappa}^4$ are independent of $i$. Let $M$ be the product of $\rho$ distinct primes of the form $\equiv 1\pmod{6}$. The number of representations of $M$ as $x^2 + xy + y^2$ with coprime integers $x,y$ with $x>y>0$ equals $2^{\rho -1}$ (see \cite{dic} par. 48, item 4). Suppose $(x_i,y_i)$ is such a pair. Choose as roots of $F_i$ the six integers $\pm x_i, \pm y_i,\pm(x_i+y_i)$ for $i=1, 2, \dots, 2^{\rho -1}$. We have $x_i^2 + y_i^2 + (x_i+y_i)^2 = 2M$ and $x_i^4 + y_i^4 + (x_i+y_i)^4 = M^2$ (cf. Choudhry \cite{chr}, Sec. 4). Thus both the sum of the squares and the sum of the biquadrates of the roots of $F_i$ are independent of $i$. The formulas of Newton-Girard imply $F_i(x) = x^6 - 2Mx^4 + M^2x^2-f_i$ for distinct integers $f_i$ and all $i$. We conclude that for every $\rho$ we can construct $2^{\rho - 1}$ distinct primitive polynomials of degree $6$ with integer roots, which differ by their constant terms only.

\vskip.2cm

\noindent {\bf Example 4.2.} We have
$$
7\cdot 13\cdot 19 = 1729 = x^2+xy+y^2 ~~{\rm for}~~(x,y) = (40,3), (37,8), (32,15), (25,23).
$$
Hence the polynomial $P(x)=x^6 - 2\cdot 1729x^4 + 1729^2x^2$ has simple integer roots when $26625600$, $177422400$, $508953600$ or $761760000$ is subtracted, since the corresponding polynomials equal
$$
(x^2-3^2)(x^2-40^2)(x^2-43^2),~~ (x^2-8^2)(x^2-37^2)(x^2-45^2),
$$
$$
(x^2-15^2)(x^2-32^2)(x^2-47^2),~~ (x^2-23^2)(x^2-25^2)(x^2-48^2),
$$
respectively.

\vskip.3cm

\noindent {\bf Case $m = 3$.} Take ${\rho}$ distinct primes of the form $\equiv 1\pmod{6}$. Then as above we can construct $s := 2^{\rho-1}$ distinct pairs $(x_i,y_i) \in \mathbb{Z}^2$ with $ x_i>y_i>0$ such that the sum $x_i^2 + x_iy_i+y_i^2$ is equal to the product $M$ of these primes. Consider the triples
$$
(M+x_i(y_i-x_i), -M+y_i(y_i-x_i), x_i^2-y_i^2)\ \ \ (i=1,2, \dots, s).
$$
Each triple has sum $0$ and sum of squares
$$
2M^2 -2M(x_i^2-2x_iy_i+y_i^2) + 2x_i^4-2x_i^3y_i-2x_iy_i^3+2y_i^4.
$$
Using that $M= x_i^2 + x_iy_i +  y_i^2$, we obtain that the sums of squares equal $2M^2$. Of course, this is also true for the opposite triples
$$
-(M+x_i(y_i-x_i)),\ \ M-y_i(y_i-x_i),\ \ y_i^2-x_i^2\ \ \ (i=1,2, \dots, s).
$$
Thus the polynomial $x^3-M^2x$ has simple integer roots if $0$ or
$$
(M+x_i(y_i-x_i))(M-y_i(y_i-x_i))(x_i^2-y_i^2)
$$
is added or subtracted, for $i=1,2, \dots, s$.

\vskip.2cm

\noindent {\bf Example 4.3.}  We start again from the pairs
$$
(x,y) = (40,3),\ \ (37,8),\ \ (32,15), \ \ (25,23)
$$
from Example 4.2 which all satisfy $M=x^2 + xy + y^2 = 1729$. According to the above rules they lead to the nine triples
$$
(-1729, 0, 1729),\ (\pm 1840, \mp 249, \mp 1591),\ (\pm 1961, \mp 656, \mp 1305),
$$
$$
(\pm 1984, \mp 1185, \mp 799),\ (\pm 1775,  \mp 96, \mp 1679),
$$
which all have sum $0$ and sum of squares $2\cdot 1729^2 = 5978882$. Thus the polynomial $P(x) = x^3 - 1729^2x$ has simple integer roots when one from
$$
0,\ \ \ \pm 728932560,\ \ \ \pm 1678772880,\ \ \ \pm 1878480960,\ \ \ \pm 286101600
$$
is added. Namely, we get the polynomials $(x-1729)x(x+1729)$,
$$
(x\pm 1840)(x\mp 249)(x\mp 1591),\ (x\pm 1961)(x\mp 656)(x\mp 1305),
$$
$$
(x\pm 1984)(x\mp 1185)(x\mp 799),\ (x\pm 1775)(x\mp 96)(x\mp 1679),\ 
$$
respectively.

\section{Standard pairs of the first or second kind}
\label{sec12}

In this section we return to the original problem on equation \eqref{maineq} subject to \eqref{fx} and show by the help of examples that all cases of the first or second kind which are not excluded may indeed occur. Suppose the equation $f(x)=g(y)$ with $f(x),g(x) \in \Q[x]$ has infinitely many solutions $(x,y)\in \Q^2$ with a bounded denominator. According to Theorem \ref{lembt} we have $f=\varphi( F(\kappa))$ and $g=\varphi( G(\lambda))$, where $\kappa(x),\lambda(x)\in\Q[x]$ are linear polynomials, $\varphi(x)\in\Q[x]$, and $F(x),G(x)$ form a standard pair over $\Q$ such that the equation $F(x)=G(y)$ has infinitely many rational solutions with a bounded denominator. In the sequel we suppose that $f=\varphi( F)$ and $g=\varphi(G)$. The results then extend to all equations similar to the equation $\varphi(F(x))=\varphi(G(y))$, in particular to the original equation $f(x)=g(y)$.

Let $\varphi(x) = p_0(x-p_1) \cdots (x-p_s)$ with $p_0\in\Q$ ($p_0\neq 0$), $p_1,\dots,p_s\in\C$. Then $f(x) = p_0(F(x)-p_1) \cdots (F(x)-p_s)$. Since $f$ has only simple rational roots, $p_i$ is in fact rational, $F(x)-p_i$ has only simple rational roots and is a $\pte_s$ component of $f$ for $i=1,\dots, s$. We assume that $(F,G)$ is a standard pair of the first or second kind and consider successively the cases $\deg(F)=1$, $\deg(F)=2$ and $\deg(F)>2$. As we shall see, in each case $\deg(\varphi)$ can attain any positive integer value, hence $\deg(f),\deg(g)$ can each be arbitrarily large. 

\vskip.2cm

\noindent {\bf Case $\deg(F)=1$.} The standard pair is of the first kind and we may assume that $F(x)=x$. Then $f = \varphi$. Hence for every $X \in \Z$ equation $F(x) = G(y)$ has as solution $(x,y) = (G(X),X)$. Thus equation $f(x)=g(y)$ has also solution $(x,y) = (G(X),X)$ for every $X\in \Z$. Here the choice of $G(y)\in\Q[y]$ is free, $\deg(f)\mid\deg(g)$ and $\deg(f)$, $\deg(g)$ can be arbitrarily large. 

\vskip.2cm

\noindent {\bf Example 5.1.}
For every set of nonzero rationals $\{a_1,a_2,\dots, a_k\}$ the equation
$$
(x-a_1)(x-a_2)\cdots (x-a_k)=(G(y)-a_1) (G(y)-a_2)\cdots (G(y)-a_k),
$$
for $G(y)\in\Q[y]$ with $\deg(G)\geq 1$ arbitrary, has solutions $(x,y) = (G(X),X)$ $(X\in \Z)$. Here $\varphi(x)=f(x)$, $F(x)=x$. Writing $n=\deg(G)$, we have $\deg(f)=k\mid nk=\deg(g)$, where $k$ and $n$ can be arbitrary.

\noindent
{\bf Example 5.2.} We start out from two triples from Example 4.3,
$(1840, -249, -1591)$ and $ (1961, -656,-1305) $
both having sum zero and equal sums of squares. Let
\begin{equation}
\label{negyzetes}
f(x)=(x^2-1840^2)(x^2-249^2)(x^2-1591^2)(x^2-1961^2)(x^2-656^2)(x^2-1305^2).
\end{equation}
Since $v(x)\pm 728932560$ and $v(x)\pm 1678772880$ with $v(x)=x^3-1729^2 x$ are given by
$$
(x\pm 1840)(x\mp 249)(x\mp 1591),\ (x\pm 1961)(x\mp 656)(x\mp 1305),
$$
respectively, we see that
$$
f(x)=(v(x)^2-728932560^2)(v(x)^2-1678772880^2).
$$
Hence, putting $F(x)=v(x), G(y)=y$,
$$
g(x) = \varphi(x)=(x^2-728932560^2)(x^2-1678772880^2),
$$
we obtain
that the equation $f(x)=g(y)$ has infinitely many integral solutions, given by $(x,y)=(X,v(X))$ $(X\in\Z)$.
Here both $f$ and $g$ have only simple rational roots.

\vskip.2cm

\noindent {\bf Case $\deg(F)=2.$} Then either $F(x)=x^2$ or $F(x)=\alpha x^2+\beta x+\gamma$, $G(y)=y^q$ with $\alpha\neq 0$, $q \in \Z_{>0}$. In the latter case we first use that $F(x) \simeq x^2+c$ for some $c \in \Q$. Here $p=0$, $q=1$, $\deg(v)=2$ if $(F,G)$ is of the first kind and $\deg{v}=0$ if $(F,G)$ is of the second kind. Subsequently we replace $\varphi(x)$ by $\varphi(x-c)$ by which we get $F(x)=x^2, G(y) = y^q-c$. Thus we may choose $F(x) = x^2$ in both cases. 

We obtain that $f(x)$ is of the form $$\varphi(F(x))= p_0(x^2-p_1) \cdots (x^2-p_s)$$ has only simple rational roots. It follows that $p_1, p_2, \dots, p_s$ are squares of distinct rational numbers and that the roots $\pm b_1, \pm b_2, \dots, \pm b_s$ of $f$ are symmetric around $0$. Furthermore, $g(y) = p_0(G(y)-b_1^2) \cdots (G(y)-b_s^2)$. By Theorem \ref{lembt} the equation $x^2=G(y)$ has to have infinitely many rational solutions $x,y$ with a bounded denominator. Let $X_i,Y_i$ $(i=1,2,\dots)$ be such solutions. By the main result of  LeVeque \cite{lev2} (for the effective version see Brindza \cite{bri}) we obtain that the polynomial $G$ can have at most two roots of odd multiplicities. It follows that the equation $f(x)=g(y)$ has infinitely many rational solutions $(x,y)=(X_i,Y_i)$ $(i=1,2,\dots)$ with a bounded denominator. Writing $n=\deg(G)$, we have $\deg(f)=2s\mid 2ns =2\deg(g)$. In this case $s$ and $n$ can be arbitrary, and hence $\deg(f)$, $\deg(g)$ may be arbitrarily large.

\vskip.2cm

\noindent {\bf Example 5.3.}  Let
$$
F(x) = x^2,\ G(y) = yv^2(y),\ \varphi(x) = (x-b_1^2)\cdots (x-b_s^2)
$$
for some $v(y)\in\Q[y]$ and distinct positive rationals $b_1, b_2, \dots, b_s$. Then
$$
f(x)=(x-b_1)(x+b_1)\cdots(x-b_s)(x+b_s), g(x)=(G(y)-b_1^2)\cdots (G(y)-b_s^2),
$$
and $f(x)=g(y)$ has solutions $(Xv(X^2), X^2)$ for every $X \in \Z$.

\vskip.2cm

\noindent {\bf Example 5.4.}  Let
$$
F(x) = x^2,\ G(y) =(2y^2-1)v^2(y),\ \varphi(x) = (x-b_1^2)\cdots (x-b_s^2)
$$
for some $v(y)\in\Q[y]$ and distinct positive rationals $b_1, b_2, \dots, b_s$. Let $(X_i)_{i=1}^{\infty}$ be distinct integers such that $2Y_i^2 -1 =X_i^2$ for integers $Y_i$. Then
$$
f(x)=(x-b_1)(x+b_1)\cdots(x-b_s)(x+b_s), g(x)=(G(y)-b_1^2)\cdots (G(y)-b_s^2),
$$
and $f(x)=g(y)$ has solutions $(X_iv(Y_i), Y_i)$ for $i=1,2, \dots$.

\vskip.2cm 

\noindent Note that Examples 5.3 and 5.4 are generalizations of Examples 1.1 and 1.2, respectively. See also Remark 3.2.

%
%
%
%
\vskip.2cm
\noindent {\bf Case $\deg(F)>2$.} Here either $F(x) = x^q$ for some $q>2$ or $G(x)=x^q$ for some positive integer $q$.

If $F(x)=x^q$, then $f(x) = p_0(x^q-p_1)(x^q-p_2) \cdots (x^q-p_s)$ has simple rational roots which implies $q\leq 2$, but since $\deg(F)>2$ this is not possible. If $G(x)=x^q$, then from Table 1 we see that either $F(x)=\alpha x^p v(x)^q$ with $0\leq p<q$, $(p,q)=1$ (if $(F,G)$ is of the first kind), or $q=2$ and $F(x)=(\alpha x^2+\beta) v(x)^2$ (if $(F,G)$ is of the second kind). Since $f$ has simple rational roots, $f' = \varphi'(F)F'$ has only simple real roots and therefore $F'$ has only simple real roots. So $q \leq 2$ and in view of $\deg(F)>2$, we have only the following possibilities:
\begin{itemize}
\item[a)] $G(x)=x$ and $F(x) = \alpha v(x)$ has only simple rational roots.
\item[b)] $G(x)=x^2$ and $F(x)$ is $\alpha x v(x)^2$ or $(\alpha x^2+\beta)v(x)^2$.
\end{itemize}
In any case, $F$ is a $\pte_s$ component of $f$. Observe in particular that in case a) we have $\varphi=g$, so 
that it also suffices that $F$ is a $\pte_s$ component of $f$, thus the roots of $f$ form an ideal $\pte_s$ set.
In case a) we clearly have $\deg(g)\mid\deg(f)$, while in case b), obviously $\deg(g)\mid 2\deg(f)$ holds. From the examples it will be clear that the degree of $\varphi$ can be arbitrary, hence the degrees of $f,g$ can be arbitrarily large.

\vskip.2cm

\noindent {\bf Example 5.5.} First we give an example for a). We start from two triples from Example 4.3,
$$
(-1729, 0, 1729),\ (1840, -249, -1591),
$$
both having sum $0$ and sum of squares $2\cdot 1729^2$. So letting
$$
f(x)=(x+1729)x(x-1729)(x-1840)(x+249)(x+1591)
$$
and $g(y)=y(y-728932560)$, we have 
$
\varphi=g, F(x)=x^3-1729^2x, \\G(y)=y.
$
The equation $f(x)=g(y)$ has solution $(x,y)=(X,F(X))$ for every $X\in\Z$. Also $g$ has only simple integer roots.

\vskip.2cm 

In case b), we give examples to show that both possibilities for the choice of $F(x)$ are possible. Recall that in both cases we have $G(y)=y^2$.

\vskip.2cm

\noindent
{\bf Example 5.6.} $F(x)=\alpha x v(x)^2$. Let
$$
f(x)=(x-249^2)(x-1591^2)(x-1840^2)(x-656^2)(x-1305^2)(x-1961^2).
$$
Observe that we simply wrote $x$ in place of $x^2$ in \eqref{negyzetes}. Then by
$$
(x-1840^2)(x-249^2)(x-1591^2)=x(x-1729^2)^2-728932560^2,
$$
$$
(x^2-1961^2)(x-656^2)(x-1305^2)=x(x-1729^2)^2-1678772880^2
$$
we obtain that setting
$$
F(x)=x(x-1729^2)^2,\ \ \ \varphi(x)=(x-728932560^2)(x-1678772880^2),
$$
and
$$
g(y)=(y^2-728932560^2)(y^2-1678772880^2),
$$
the equation $f(x)=g(y)$ has infinitely many solutions given by $(x,y)=(X^2,X(X^2-1729^2))$ $(X\in\Z)$. Again $g$ has only simple rational roots.

\vskip.2cm

\noindent
{\bf Example 5.7.} $F(x)=(\alpha x^2+\beta)v(x)^2$. We use data from Example 4.1. We start from
$
1105=33^2+4^2=32^2+9^2,
$
and put
$$
f(x)=26^2(x^2-33^2)(x^2-4^2)(x^2-32^2)(x^2-9^2).
$$
Then $f$ has only simple rational roots. Put
$$
F(x)=26x^2(x^2-1105), \varphi(x)=(x+26\cdot(33\cdot 4)^2)(x+26\cdot(32\cdot 9)^2), g(y) = \varphi(y^2).
$$
A Magma \cite{magma} calculation shows that the equation
$$
26(x^2-1105)=Y^2
$$
has solutions $(x,y)=(X_i,Y_i)$ $(i \in \Z)$, with $(X_1,Y_1) = (247, -1248)$, $(X_2,Y_2)=(117,572)$, 
and
$$
(X_i,Y_i)=102(X_{i-1},Y_{i-1})-(X_{i-2},Y_{i-2})\ \ \ (i \in \Z_{>2}).
$$
So the equation $f(x)=g(y)$ has infinitely many integral solutions $(X,Y) = (X_i,X_iY_i)$. 

\section{Both $f$ and $g$ have only simple rational roots and \\$(F,G)$ is of the first or second kind}
\label{sec12fg}

In this section we consider equation \eqref{maineq} with both $f$ and $g$ having only simple rational roots. Without loss of generality we may assume $\deg(f) \leq \deg(g)$, hence $\deg(F) \leq \deg(G)$. We again assume $f= \varphi(F), g = \varphi(G)$. 

\begin{theorem}
\label{kinds12}
Let $f(x), g(x) \in \Q[x]$, both having only simple rational roots. Suppose that the equation $f(x)=g(y)$ has infinitely many rational solutions $x,y$ with a bounded denominator and that the corresponding standard pair $(F(x), G(x))\in \Q[x]$ is of the first or second kind. Then we can choose $F, G, \varphi$ such that one of the following items holds:

\vskip.1cm

\noindent {\rm \bf 1.} $\deg(f)\mid \deg(g)$, there exist $p_0 \in \Q, p_0 \neq 0$ and distinct $p_1, p_2, \dots, p_s \in \Q$ such that 
\begin{equation}
\label{q1}
f(x) =p_0 \prod_{i=1}^s (x-p_i),~~ g(y) = p_0\prod_{i=1}^s (G(y)-p_i),
\end{equation} 
$F(x)=x$ and $(G(y)-p_i)_{i=1}^s$ forms a $\pte_{n}$ set, where $n = \deg(G)$, for every $X \in \Z$ the equation $f(x)=g(y)$ has solution $(x,y) = (G(X),X)$.

\vskip.1cm

\noindent
{\rm \bf 2.} $\deg(f)\mid 2\deg(g)$, 
there exist $q_0\in\Q, q_0 \not= 0$ and distinct $q_1, q_2, \dots, q_s \in \Q_{>0}$ such that 
\begin{equation}
\label{q2}
f(x) =q_0 \prod_{i=1}^s (x-q_i)(x+q_i),~~ g(y) = q_0\prod_{i=1}^s (G(y)-q_i^2),
\end{equation}
$F(x)=x^2$, $G(y)\in\Q[y]$ has at most two roots of odd multiplicities, $(G(y)-q_i^2)_{i=1}^s$ forms a $\pte_{n}$ set, the equation $x^2 = G(y)$ has infinitely many rational solutions $(x,y)=(X_i,Y_i)$ $(i=1,2,\dots)$ with a bounded denominator and the equation $f(x)=g(y)$ has solutions $(x,y) = (X_i,Y_i)$ $(i=1,2,\dots)$ too.
\end{theorem}

\begin{proof} By Lemma \ref{btcons} we know that $\deg(F)\leq 2$.

If $\deg(F)=1$, then (using the notation of Table 1) we have $p=0$, $q=1$, $F(x)=x$, $G = v$, $f = \varphi$ and $G(y)-p_1, \dots, G(y) - p_s$ form a $\pte_n$ set. Then $f,g$ are as in \eqref{q1}. For every $X\in\Z$ there is a solution $(x,y) = (G(X),X)$ to the equation. This case is covered by {\bf 1.}

If $\deg(F)=2$ then we may assume $F(x)=x^2$ according to the argument given in the preceding section.
As before, we see that $p_1, p_2, \dots, p_s$ are squares in $\Q$. Let $p_i=q_i^2$ for $i=1,2, \dots, s.$ Then $f,g$ are as in \eqref{q2} where $G(y) -q_1^2, \dots, G(y) - q_s^2$ form a $\pte_{n}$ set. Further, by Theorem \ref{lembt} we know that the equation $x^2 = G(y)$ has infinitely many solutions in rationals $x,y$ with a bounded denominator. Clearly, these solutions will be solutions to the original equation, too. The main result of LeVeque \cite{lev2} shows that $G(y)$ can have at most two roots of odd multiplicities. This is covered by {\bf 2.}
\end{proof}

\vskip.2cm


The following example illustrates that the results in Section \ref{secpte} imply that there are instances for $\deg(F)=1$ with $\deg(G) \in \{3,4,6\}$ and arbitrary $\deg(\varphi)$ (and hence arbitrarily large $\deg(f),\deg(g)$ as well). This is obvious for $\deg(G)=2$, cf. Example 5.1.

\vskip.2cm

\noindent {\bf Example 6.1.} (Cf. Examples 4.1, 4.2 and 4.3.)
For $\deg(G)=4$ choose $G(y) = y^4 - 1105y^2, F(x)=x$ and
$$
f(x) = \varphi(x)=(x+17424)(x+82944)(x+138384)(x+304704).
$$
Then $g(y)=\varphi(G(y))$ is given by
$$
(y^2-33^2)(y^2-4^2)(y^2-32^2)(y^2-9^2)(y^2-31^2)(y^2-12^2)(y^2-24^2)(y^2-23^2).
$$
For every integer $X$ we obtain a solution $(x,y) = (G(X),X)$ of equation \eqref{maineq}. It is obvious that the roots of $g$ are symmetric around $0$.

Similarly, for $\deg(G)=6$ choose
$$
G(y) = y^6-2\cdot 1729y^4 +1729^2y^2, F(x)=x,
$$
$$
f(x)= \varphi(x) = (x-26625600, (x-177422400)(x-508953600)(x-761760000).
$$
Then $g(y)=\varphi(G(y))$ is the polynomial $g(y) = \prod_{a \in T} (y^2-a^2)$ with $T=\{3, 40, 43, 8, 37, 45, 15, 32, 47, 23, 25, 48\}$.
Again, for every integer $X$ we obtain a solution $(x,y) = (G(X),X)$ of equation \eqref{maineq}.

Finally, for $\deg(G) = 3$ let $G(y) = y^3 - 1729^2 y, F(x)=x$,
\begin{multline*}
f(x) =\varphi(x)=\\
x(x^2 - 728932560^2)(x^2 - 1678772880^2)(x^2 - 1878480960^2)(x^2 - 286101600^2),
\end{multline*}
and so
$g(y) = y \prod_{a \in T} (y^2-a^2)$ with
$$
T=\{1729, 1840, 249, 1591, 1961, 656, 1305, 1984, 1185, 799, 1775, 96, 1679\}.
$$
Also in this case, for every integer $X$ we obtain a solution $(x,y) = (G(X),X)$ of equation \eqref{maineq}.

\vskip.2cm

Now we turn to the case $\deg(F)=2$. Our examples will show that in all the possible cases, $\deg(\varphi)$ can be arbitrary (whence $\deg(f)$ and $\deg(g)$ can be arbitrarily large). Note that $\deg(G)\geq 2$. Interchanging the roles of $F$ and $G$ if necessary if $\deg(F)=\deg(G)=2$, by the analysis in the cases $\deg(F)=2$ and $\deg(F)>2$ in Section \ref{sec12}, we see that $G$ is of the form $\alpha x v(x)^2$ or $(\alpha x^2+\beta) v(x)^2$.
The case where $G(x)$ is of the shape $\alpha x v(x)^2$ is covered by Example 5.6.  
The other possibility is that $G(x)$ is of the shape
$(\alpha x^2+\beta)v(x)^2$. 
The next example, which is another generalization of Example 1.2, treats this situation with $\deg(v)=0$.

\vskip.2cm 

\noindent
{\bf Example 6.2.} Suppose that the equation $x^2 = ay^2 +b$ with $a,b\in\Z$, $ab\neq 0$ has solutions $(X_i,Y_i)_{i=1}^{\infty} \in \mathbb{Z}^2$. Let $s\geq 1$, $F(x) = x^2$, $G(y) = ay^2+b$, $\varphi(x) = \prod_{i=1}^s (x-X_i^2)$. Then we have
$$
f(x)=\prod_{i=1}^s (x^2 - X_i^2),\ \ \ g(y)= \prod_{i=1}^s (ay^2+b-X_i^2) = a^s \prod_{i=1}^s (y^2-Y_i^2).
$$
So $f(x)$ and $g(y)$ both have simple rational roots. Further, clearly, the equation $f(x)=g(y)$ has as solutions $(X_i,Y_i)$ for all $i=1,2,\dots$. We see that $\deg(\varphi)$ can be arbitrary, hence $\deg(f)$ and $\deg(g)$ can be arbitrarily large.

%

%
%
%

\section{Standard pairs $F,G$ of the third or fourth kind}
\label{sec23}

To handle the cases corresponding to standard pairs of the third and fourth kind, we apply the following result.

\begin{lemma}
\label{lemdickson}
Let $a_1,\dots,a_N$ be distinct rationals, and assume that with some rational numbers $u_1,u_2,v_1,v_2,b$ with $u_1v_1b\neq 0$ we have
\begin{equation}
\label{eq1}
u_1D_N(x,b)+u_2=(v_1x+v_2-a_1)\cdots (v_1x+v_2-a_N),
\end{equation}
where $D_N(x,b)$ is the $N$-th Dickson polynomial with parameter $b$. Then $N\in\{1,2,3,4,6\}$.
\end{lemma}

\noindent
Note that $N\leq 12$ is already proved in \cite{hpt} (see the proof of Theorem 2.3 there). However, in this paper we need a more precise statement. To keep the presentation self-contained, we include the complete proof of the statement.

For appropriate choices of the parameters the cases $N\in\{1,2,3,4,6\}$ are possible. In Theorem \ref{remar} we describe these cases completely.

\vskip.2cm


\begin{proof}[Proof of Lemma \ref{lemdickson}]
Writing $w_i=(v_2-a_i)/v_1$ $(i=1,\dots,N)$ and $u=u_2/v_1^N$, dividing both sides of \eqref{eq1} by $v_1^N$ and using that $D_N$ is monic, we get the similar equation
\begin{equation}
\label{eq2}
D_N(x,b)+u=(x+w_1)\dots (x+w_N).
\end{equation}
Here $u\in\Q$ and $w_1,\dots,w_N$ are distinct rationals.

Applying the well-known identity
$$
D_N\left(y+\frac{b}{y},b\right)=y^N+\left(\frac{b}{y}\right)^N
$$
to \eqref{eq2}, we obtain
\begin{equation}
\label{eq3}
y^{2N}+uy^N+b^N=\prod\limits_{i=1}^N (y^2+w_iy+b)\ .
\end{equation}
Write $\zeta,\xi$ for the roots of the polynomial $Y^2+uY+b^N$. Clearly, $\zeta,\xi$ are algebraic numbers of degrees at most two. Further, 
$b\neq 0$ yields $\zeta\xi\neq 0$. Also observe that $\zeta\neq \xi$, since the numbers $w_i$ in \eqref{eq2} are distinct. If $u=0$, then the roots of the left-hand side of \eqref{eq3} are given by
\begin{equation} \label{eta}
\eta^j\sqrt{b}\ \ \ (j=0,1,\dots,2N-1),
\end{equation}
where $\sqrt{b}$ denotes one of the (complex) squareroots of $b$, and $\eta$ is a primitive $2N$-th root of unity. In view of the right-hand side of \eqref{eq3}, we see that the numbers \eqref{eta} are algebraic numbers of degrees at most two. Hence
$
\varphi(2N)=\deg(\eta)\leq 4.
$
This implies $N \in \{1,2,3,4,6\}.$

So from this point on, we assume $\zeta+\xi=u\neq 0$. Then the roots of the polynomial on the left hand side of \eqref{eq3} are given by
$$
\zeta_0 \varepsilon^i\ \ \text{and}\ \ \xi_0 \varepsilon^i\ \ \ (i=0,1,\dots,N-1),
$$
where $\zeta_0$ and $\xi_0$ are $N$-th roots of $\zeta$ and $\xi$, respectively, and $\varepsilon$ is a primitive $N$-th root of unity. 
Since these are the roots of the polynomial on the right hand side of \eqref{eq3}, they are distinct algebraic numbers of degrees at most two. In particular, $\zeta_0$ and $\zeta_0\varepsilon$ are at most quadratic algebraic numbers, so the degree of $\varepsilon$ is at most four. Hence $\varphi(N)\leq 4$, and we obtain $N\in\{1,2,3,4,5,6,8,10,12\}$.

To refine the restriction for $N$, we need a more careful consideration. Write $\zeta_1:=\zeta_0\varepsilon$. Then we see that
\begin{equation}
\label{eq4}
\varepsilon=\frac{\zeta_1}{\zeta_0}
\end{equation}
belongs to the number field $K:=\Q(\zeta_0,\zeta_1)$. Observe that if $\zeta_0\in\Q(\zeta_1)$, then $\varepsilon$ is (at most) quadratic, yielding $\varphi(N)\leq 2$, and our claim follows. So we may assume that $\deg(K)=4$, and also that $K=\Q(\varepsilon)$ and that $\zeta_0$ is quadratic. Denoting its algebraic conjugate by $\bar\zeta_0$, we have
$$
(\overline{\zeta_0})^N={\overline{\zeta_0^N}}=\bar\zeta=\xi.
$$
Therefore, without loss of generality we may assume that $\xi_0=\overline{\zeta_0}$ holds, in particular, that $\zeta_0$ and $\xi_0$ belong to the same quadratic subfield of $K$. From this point on, we shall use this assumption. We deal with the remaining cases in turn. For the calculations we used Magma \cite{magma}.

If $N=5,10$, then $K$ is defined by $x^4+x^3+x^2+x+1$. The only quadratic subfield of $K$ is given by $T_1:=\Q(\sqrt{5})$. So now $\zeta_0,\xi_0\in T_1$. Recall that $\zeta_0$, hence also $\xi_0$ is not rational. However, the (unique) factorization of 
\begin{equation}
\label{px}
P(x):=x^{2N}+ux^N+b^N=(x^N-\zeta_0^N)(x^N-\xi_0^N)
\end{equation} 
(into irreducible factors) in $T_1[x]$ contains both for $N=5$ and for $N=10$ the factors
$$
x^2+(3-\sqrt{5})\zeta_0x+\zeta_0^2~~{\rm and}~~x^2+(3-\sqrt{5})\xi_0x+\xi_0^2.
$$
Here the constant terms of the quadratic factors are not equal. Indeed, otherwise $\zeta_0^2=\xi_0^2$ would imply $\zeta_0=\pm \xi_0$, whence $\zeta=\pm \xi$, which is excluded. Hence we see that \eqref{eq3} is not possible in these cases.

Let now $N=8$. Then $K$ is defined by $x^4+1$. The number field $K$ has three quadratic subfields, namely $T_2=\Q(i)$, $T_3=\Q(\sqrt{2})$ and $T_4=\Q(i\sqrt{2})$.  Following the argument given above for the factorization of $P(x)$ defined by \eqref{px} we get that
\begin{itemize}
\item $x^2+i\zeta_0^2$ and $x^2+i\xi_0^2$ are factors of $P(x)$ in $T_2[x]$,
\item $x^2+\zeta_0^2$ and $x^2+\xi_0^2$ are factors of $P(x)$ in $T_3[x]$ and $T_4[x]$,
\end{itemize}
assuming that $\zeta_0,\xi_0\in T_2,T_3,T_4$, respectively. In all cases the constant terms of the quadratic factors are not the same. So $N=8$ is also impossible.

Finally, let $N=12$. Then $K$ is defined by $x^4-x^2+1$. The number field $K$ has three quadratic subfields, namely $T_5=\Q(i)$, $T_6=\Q(\sqrt{3})$ and $T_7=\Q(i\sqrt{3})$. Now, similarly as before, for the factorization of $P(x)$ given by \eqref{px} we obtain that
\begin{itemize}
\item $x^2+\zeta_0x+\zeta_0^2$ and $x^2+\xi_0x+\xi_0^2$ are factors of $P(x)$ in $T_5[x]$,
\item $x^2+\zeta_0^2$ and $x^2+\xi_0^2$ are factors of $P(x)$ in $T_6[x]$ and $T_7[x]$,
\end{itemize}
assuming that $\zeta_0,\xi_0\in T_5,T_6,T_7$, respectively. \\
Again, in all cases we observe that the constant terms of the quadratic factors are not identical. So $N=12$ is excluded, too.
\end{proof}

\begin{theorem} \label{remar}
Let $N \in \{3,4,6\}$. For any $w_1,w_2 \in \mathbb{Q}$ we can choose $w_3, \dots, w_N, b,u \in \mathbb{Q}$ such that \eqref{eq2} holds. On the other hand, this provides the only solutions of equation \eqref{eq2}.
\end{theorem}

The cases $N=1$ and $N=2$ are trivial. Indeed, for $N=1$ we have $D_1(x,b)=x$, so $w_1=u$ can be any rational number. Further, for $N=2$
we have $D_2(x,b)=x^2-2b$, whence $w_1+w_2=0$,  $w_1w_2 = -2b+u$. Therefore all cases of \eqref{eq2} are given by
$$
(x+w_1)(x-w_1) = D_2(x,b) + (2b - w_1^2),
$$
i.e. with $u=2b - w_1^2$ for arbitrary $b, w_1 \in \Q$.  

\begin{proof}[Proof of Theorem \ref{remar}]
We consider the possibilities in turn.

{\bf The case $N=3$}. We have 
\begin{equation} \label{d3}
D_3(x,b)=x^3-3bx,
\end{equation} 
hence 
$$
w_1+w_2+w_3=0, w_1w_2+w_1w_3+w_2w_3 = -3b, w_1w_2w_3 = u.
$$
This gives 
\begin{equation} \label{3bu} 
w_3 = -w_1-w_2, b= (w_1^2 +  w_1w_2 + w_2^2)/3, u=  -w_1^2w_2 -w_1w_2^2.
\end{equation} 
Thus we have for any $w_1, w_2 \in \mathbb{Q}$ that 
\begin{equation} \label{d3enu}
(x+w_1)(x+w_2)(x-w_1-w_2) = D_3(x, b) +u
\end{equation}
and this provides all possibilities for \eqref{eq2}.

{\bf The case $N=4$}. We have 
\begin{equation} \label{d4}
D_4(x,b)=x^4-4bx^2+2b^2.
\end{equation}
This implies
$$
w_1+w_2+w_3+w_4=0
$$
and
$$
w_1w_2w_3+w_1w_2w_4+w_1w_3w_4+w_2w_3w_4=0.
$$
It follows that $$0 = w_1w_2w_3-(w_1w_2+w_1w_3+w_2w_3)(w_1 +w_2+w_3) = -(w_1+w_2)(w_1+w_3)(w_2+w_3).$$
We assume, without loss of generality, 
\begin{equation} \label{choice4w}
w_1+w_3=0, ~~{\rm hence} ~~w_2+w_4=0.
\end{equation} 
Further comparison of coefficients gives
\begin{equation} \label{value4b}
b=-\frac{w_1w_2+w_1w_3+w_1w_4+w_2w_3+w_2w_4+w_3w_4}{4} =  \frac{w_1^2+w_2^2}{4}
\end{equation}
and
\begin{equation} \label{value4u}
u=w_1w_2w_3w_4 -2b^2= w_1^2w_2^2 - \frac18(w_1^2 +w_2^2)^2= -\frac18(w_1^4 -6w_1^2w_2^2+w_2^4).
\end{equation}
For any $w_1,w_2 \in \Q$ and $b,u$ chosen as above we have
\begin{equation} \label{d4enu}
(x+w_1)(x-w_1)(x+w_2)(x-w_2) = D_4(x,b) +u
\end{equation} 
and this provides all possibilities for \eqref{eq2}. 

{\bf The case} $N=6$ is the most involved one. We have 
\begin{equation}
\label{d6}
D_6(x,b)=x^6-6bx^4+9b^2x^2-2b^3.
\end{equation}
On the other hand, the roots of the polynomial on the left hand side of \eqref{eq3} are given by
$$
\pm \zeta_0,\ \pm \zeta_0\varepsilon,\ \pm \zeta_0\varepsilon^2,\ \pm \xi_0,\pm \xi_0\varepsilon,\ \pm \xi_0\varepsilon^2,
$$
where $\varepsilon$ is a primitive sixth root of unity (i.e. a root of $x^2-x+1$), and either $\zeta_0,\xi_0\in\Q$, or they are conjugated quadratic algebraic numbers from the field $K=\Q(\varepsilon)$. Anyhow, the factorization of the polynomial on the right hand side of \eqref{eq3} over $K$ reads as
\begin{multline*}
(y-\zeta_0)(y+\zeta_0)(y-\zeta_0\varepsilon)(y+\zeta_0\varepsilon)(y-(1-\varepsilon)\zeta_0)(y+(1-\varepsilon)\zeta_0)\cdot\\
\cdot(y-\xi_0)(y+\xi_0)(y-\xi_0\varepsilon)(y+\xi_0\varepsilon)(y-(1-\varepsilon)\xi_0)(y+(1-\varepsilon)\xi_0).
\end{multline*}
Note that the (algebraic) conjugate of $\varepsilon$ is $1-\varepsilon$. Hence we immediately get that the right hand side of \eqref{eq3} is given by
\begin{multline*}
(y^2-(\zeta_0+\xi_0)y+\zeta_0\xi_0)(y^2+(\zeta_0+\xi_0)y+\zeta_0\xi_0)\cdot\\
\cdot(y^2-(\zeta_0\varepsilon+\xi_0(1-\varepsilon))y+\zeta_0\xi_0)(y^2+(\zeta_0\varepsilon+\xi_0(1-\varepsilon))y+\zeta_0\xi_0)\cdot\\
\cdot(y^2-(\zeta_0(1-\varepsilon)+\xi_0\varepsilon)y+\zeta_0\xi_0)(y^2+(\zeta_0(1-\varepsilon)+\xi_0\varepsilon)y+\zeta_0\xi_0).
\end{multline*}
Here all the above quadratic polynomials have rational coefficients. The coefficients of $y$ are just the numbers $w_i$ from \eqref{eq3} (and \eqref{eq2}). Observe that (by choosing an appropriate indexing) we have 
\begin{equation}
\label{felt1}
w_3 = w_1+w_2, \ w_4=-w_1,\ w_5=-w_2,\ w_6=-w_3
\end{equation}
in \eqref{eq2}. Put $W = w_1^2 + w_1w_2 + w_2^2$. A simple calculation yields that
\begin{multline*}
(x+w_1)(x+w_2)(x+w_3)(x+w_4)(x+w_5)(x+w_6)=\\
=x^6-2Wx^4+W^2x^2-w_1^2w_2^2(w_1+w_2)^2.
\end{multline*}
Comparing the coefficients with $D_6(x,b)+u=x^6-6bx^4+9b^2x^2-2b^3+u$ we see that
\begin{equation}
\label{felt2}
b=\frac{W}{3},~~u= \frac{2W^3}{27} -w_1^2w_2^2(w_1+w_2)^2.
\end{equation}

On the other hand, for any $w_1,w_2 \in \mathbb{Q}$ we have, choosing $b$ and $u$ as in \eqref{felt2}, that
\begin{equation} \label{d6enu}
(x+w_1)(x-w_1)(x+w_2)(x-w_2)(x+w_1+w_2)(x-w_1-w_2) = D_6(x,b)+u.
\end{equation}
Thus this provides all  possibilities for \eqref{eq2} if $N=6$.
\end{proof}

We give some examples to show that for $\deg(F)=m \in \{3,4,6\}$ equation \eqref{maineq} with $f$ of the form \eqref{fx} can have infinitely many solutions $(x,y)\in\Q^2$ with a bounded denominator. For the sake of completeness, we shall do so both for the third and for the fourth kind. By the gcd condition in Table 1, $\deg(F)=m=3$ cannot occur in the latter case.

\vskip.2cm
\noindent 
{\bf Example 7.1.} Let $(F,G)$ be a standard pair of the third kind, $\deg(F)=m=3$, $\deg(G)=n=4$ and $b=7$. 
We have
$$
3b^4=3\cdot 7^4=14^2+14\cdot 77+77^2 = 23^2+23\cdot 71+71^2.
$$
We choose
$
(w_1,w_2)=(14,77),\ (23,71),
$
which, by \eqref{3bu}, gives
$
w_3=-91,\ -94,
$
respectively. Thus, by \eqref{d3enu},
\begin{multline*}
D_3(x,7^4)=(x+14)(x+77)(x-91)+14\cdot 77\cdot 91=\\
=(x+23)(x+71)(x-94)+23\cdot 71\cdot 94.
\end{multline*}
According to formula (5) of \cite{bt} we have, for all coprime positive integers $m,n$ and integers $b$,
\begin{equation} \label{dmndnm}
D_m(D_n(x,b),b^n) = D_n(D_m(x,b),b^m).
\end{equation}
Therefore the equation $F(x) := D_3(x,b^4)=D_4(y,b^3) =: G(y)$ has solutions $(x,y) = (D_4(X,7), D_3(X,7))$ for every $X\in \Z$. We obtain that
the equation
\begin{multline*}
f(x) := (x+14)(x+77)(x-91)(x+23)(x+71)(x-94)=\\
=(D_4(y,7^3)-14\cdot 77\cdot 91)(D_4(y,7^3)-23\cdot 71\cdot 94) =:g(y)
\end{multline*}
has the same solutions. Note that here we have
$$
\varphi(x)=(x-14\cdot 77\cdot 91)(x-23\cdot 71\cdot 94).
$$

\vskip.2cm

\noindent
{\bf Example 7.2.} Let $(F,G)$ be of the third kind with $\deg(F)=m=4$, $\deg(G)=n=3$ and $b=5$. 
We have $4b^3=4\cdot 5^3=4^2+22^2=10^2+20^2$. Taking $(w_1,w_2)=(4,22)$, $(10,20)$, by \eqref{choice4w}, we get $(w_3,w_4)=(-4,-22)$, $(-10,-20)$, and by \eqref{value4u}, $u=-23506$, $8750$, respectively. That is, we have
\begin{multline*}
D_4(x,5^3)=(x+4)(x-4)(x+22)(x-22)+23506=\\
=(x+10)(x-10)(x+20)(x-20)-8750.
\end{multline*}
Since, by \eqref{dmndnm}, the equation $F(x) := D_4(x,5^3)=D_3(y,5^4) =: G(y)$ has infinitely many integral solutions $(x,y)= (D_3(X,5), D_4(X,5))$ $(X\in\Z)$, we obtain that the equation
\begin{multline*}
f(x) := (x+4)(x-4)(x+22)(x-22)(x+10)(x-10)(x+20)(x-20)=\\
=(D_3(y,5^4)-23506)(D_3(y,5^4)+8750) =: g(y)
\end{multline*}
has the same solutions. Here we have
$
\varphi(x)=(x-23506)(x+8750).
$

\vskip.2cm

\noindent
{\bf Example 7.3.} The case $(F,G)$ is a standard pair of the third kind, $\deg(F)=m=6$, $\deg(G)=n=5$ and $b=7$. We have $3\cdot 7^5=211^2+211\cdot 25+25^2=196^2+196\cdot 49+49^2$. Taking $(w_1,w_2)=(211,25)$, $(196,49)$ by \eqref{felt1} we get
$$
(w_3,w_4,w_5,w_6)=(236,-211,-25,-236),\ (245,-196,-49,-245),
$$
and by \eqref{felt2}, $u=7945347009886$, $3958608139486$, respectively. That is, we have
\begin{multline*}
D_6(x,7^5)=\\
(x+211)(x+25)(x+236)(x-211)(x-25)(x-236)-7945347009886=\\
=(x+196)(x+49)(x+245)(x-196)(x-49)(x-245)-3958608139486.
\end{multline*}
Since, by \eqref{dmndnm}, the equation $F(x) := D_6(x,7^5)=D_5(y,7^6) =: G(y)$ has infinitely many integral solutions $(x,y)= (D_5(X,7), D_6(X,7))$ $(X\in\Z)$, we obtain that the equation
\begin{multline*}
f(x) := (x+211)(x+25)(x+236)(x-211)(x-25)(x-236)\cdot\\
\cdot (x+196)(x+49)(x+245)(x-196)(x-49)(x-245)=\\
=(D_5(y,7^6)+7945347009886)(D_5(y,7^6)+3958608139486) =: g(y)
\end{multline*}
has the same solutions. Here we have
$$
\varphi(x)=(x+7945347009886)(x+3958608139486).
$$

\vskip.2cm

\noindent
{\bf Example 7.4.} The case $(F,G)$ is a standard pair of the fourth kind with $\deg(F)=m=4$, $\deg(G)=n=10$ and $b=5\cdot 13=65$. We have $4b=2^2+16^2=8^2+14^2$. We take $(w_1,w_2)=(2,16)$, $(8,14)$, which, by \eqref{choice4w}, gives $(w_3,w_4)=(-2,-16)$, $(-8,-14)$ and, by \eqref{value4u}, $u=-7426$, $4094$, respectively. Thus
\begin{multline*}
D_4(x,65)=(x+2)(x-2)(x+16)(x-16)+7426=\\
=(x+8)(x-8)(x+14)(x-14)-4094.
\end{multline*}
According to formula (10) of \cite{bt} with $m=4$, $n=10$ we have for $a,b,v_1,v_2 \in \Q$ with
\begin{equation}
\label{pelll}
b^2v_1^2 +av_2^2 =4ab
\end{equation}
that 
\begin{equation} \label{d4d6} 
b^{-2}D_4(b^{-2}(v_2^5-5bv_2^3+5b^2),b) = -a^{-5}D_{10} (v_1v_2,a).
\end{equation}
Here $b=65$. We choose $a=-10\cdot 65^2$, and observe that then \eqref{pelll} becomes the (generalized) Pell equation $v_1^2 - 10v_2^2 = -2600$, and $(v_1,v_2)=(X_i,Y_i)$ given by
$(X_0,Y_0)=(-80,30)$, $(X_1,Y_1)=(280,90)$ and
$$
(X_i,Y_i)=38(X_{i-1},Y_{i-1})-(X_{i-2},Y_{i-2})\ \ \ (i\geq 2)
$$
are integral solutions to it. Thus, for any constant $c$,
$$
b^{-2}D_4(x,b)-cb^{-2}=-a^{-5}D_{10}(y,a)-cb^{-2}
$$
has infinitely many solutions $(x,y)\in\Q^2$ with a bounded denominator which are independent of $c$, and we obtain that the equation
\begin{multline*}
f(x) := b^{-4}(x-2)(x+2)(x-16)(x+16)(x-8)(x+8)(x-14)(x+14)=\\
=(-a^{-5}D_{10}(y,a)-7426b^{-2})(-a^{-5}D_{10}(y,a)+4094b^{-2}) =: g(y)
\end{multline*}
(with $a=-10\cdot 65^2$ and $b=65$) has infinitely many solutions $(x,y)\in\Q^2$ with a bounded denominator. Here we have $F(x) = b^{-2}D_4(x,b)$, $G(x) = -a^{-5}D_{10}(x,a)$, $\varphi(x)=(x-7426b^{-2})(x+4094b^{-2})$.

\vskip.2cm

\noindent
{\bf Example 7.5.} The case $(F,G)$ is a standard pair of the fourth kind with $\deg(F)=m=6$, $\deg(G)=n=10$, and $b=7\cdot 13=91$. We have $3b=16^2+16\cdot 1+1^2=11^2+11\cdot 8+8^2$. Letting $(w_1,w_2)=(16,1)$, $(11,8)$ by \eqref{felt1} we obtain
$$
(w_3,w_4,w_5,w_6)=(17,-16,-1,-17),\ (19,-11,-8,-19),
$$
and \eqref{felt2} yields $u=1433158$, $-1288442$, respectively. That is, we have
\begin{multline*}
D_6(x,91)=\\
(x+16)(x+1)(x+17)(x-16)(x-1)(x-17)-1433158=\\
=(x+11)(x+8)(x+19)(x-11)(x-8)(x-19)+1288442.
\end{multline*}
By formula (10) of \cite{bt} with $m=6$, $n=10$, if for $a,b,v_1,v_2 \in \Q$ we have
\begin{equation}
\label{pelll2}
b^3v_1^2 +av_2^2 =4ab
\end{equation}
then
\begin{equation} \label{d6d10}
b^{-3}D_6(b^{-2}(v_2^5-5bv_2^3+5b^2),b) = -a^{-5}D_{10} (v_1(v_2^2-b),a).
\end{equation}
Now $b=91$. We take $a=-14\cdot 91^3$, and observe that then \eqref{pelll2} becomes the (generalized) Pell equation $v_1^2 - 14v_2^2 = -5096$, and $(v_1,v_2)=(X_i,Y_i)$ given by
$(X_0,Y_0)=(-140,42)$, $(X_1,Y_1)=(252,70)$ and
$$
(X_i,Y_i)=30(X_{i-1},Y_{i-1})-(X_{i-2},Y_{i-2})\ \ \ (i\geq 2)
$$
are integer solutions to it. Hence, for any $c\in\Q$,
$$
b^{-3}D_6(x,b)-cb^{-3}=-a^{-5}D_{10}(y,a)-cb^{-3}
$$
has infinitely many solutions $(x,y)\in\Q^2$ with a bounded denominator which are independent of $c$. This implies that the equation
\begin{multline*}
f(x) := b^{-6}(x+16)(x+1)(x+17)(x-16)(x-1)(x-17)\cdot\\
\cdot(x+11)(x+8)(x+19)(x-11)(x-8)(x-19)=\\
=(-a^{-5}D_{10}(y,a)+1433158b^{-3})(-a^{-5}D_{10}(y,a)-1288442b^{-3}) =: g(y)
\end{multline*}
(with $a=-14\cdot 91^3$ and $b=91$) has infinitely many solutions $(x,y)\in\Q^2$ with a bounded denominator. Here we have $F(x) = b^{-3}D_6(x,b)$, $G(x) = -a^{-5}D_{10}(x,a)$, $\varphi(x)=(x+1433158b^{-3})(x-1288442b^{-3})$.

\vskip.2cm
\noindent
{\bf Remark 7.1.} We recall that if $M$ is the product of $\rho$ distinct primes of the form $\equiv 1\pmod{4}$, then the number of representations of $M$ as    
$\alpha_1^2 + \alpha_2^2$ with $\alpha_1, \alpha_2 \in \mathbb{Z}$, $\alpha_1 > \alpha_2 > 0$, $\gcd(\alpha_1, \alpha_2) = 1$ equals $2^{\rho-1}$ according to Theorem 7.5 of \cite{lev}. Similarly, the number of representations of $M$ as $x^2 + xy + y^2$ with coprime integers $x,y$, $x>y>0$ equals $2^{\rho -1}$ (see \cite{dic} par. 48, item 4).
It follows that in all the above examples $\deg(\varphi)$ can equal any $s \in \Z_{>0}$, (hence $\deg(f)$ and $\deg(g)$ can be made arbitrarily large) by choosing suitable $W$ with number of representations $\geq s$ and corresponding values $u=u_1,u_2,\dots,u_s$. 

\vskip.2cm
\noindent
{\it Proof of the first statement of Theorem 1.1.} By Lemma \ref{lemdickson} equation \eqref{maineq} with \eqref{fx} implies that $\deg(F)\in \{1,2,3,4,6\}$, if the corresponding standard pair $(F,G)$ is of the third or fourth kind. This combined with Lemma \ref{btcons} completes the proof of the first statement of Theorem \ref{caseA}. $\qed$

\vskip.2cm
\noindent {\bf Remark 7.2.} 
Let $f(x) \in \Q[x]$ have only simple rational roots and let $g(x) \in \mathbb{Q}[x]$. 
Suppose the equation $f(x)=g(y)$ has infinitely many solutions $(x,y) \in \Q^2$ with a bounded denominator. By Lemma \ref{lemdickson} we have $\deg(F) \in \{1,2,3,4,6\}$.
Put $s = \gcd(\deg(f), \deg(g))$, $m = \deg(f)/s, n= \deg(g)/s.$ Then  $ \gcd(m,n) =1$ and $m \in \{1,2,3,4,6\}$ or $n \in \{1,2\}$. 

If $m=1$ we refer to Example 5.1 to see that all pairs $n,s$ are possible. For $m=2$ Example 5.3 shows that all pairs $(n,s)$ (with $n$ odd since $m$ and $n$ are coprime) are possible. By \eqref{dmndnm} the equation $D_m(x,b^n) = D_n(y,b^m)$ with $\gcd(m,n)=1$ has infinitely many solutions in integers $(x,y)$ for any integer $b$. If $m=4$, we proceed as in Example 7.2 (where $s=2$) using a $b$ which is the product of sufficiently many distinct primes $\equiv 1\pmod{4}$. If $m=3$ or $m=6$, then we proceed as in Examples 7.1 or 7.3 (where $s=2$ too) using a $b$ which is the product of sufficiently many distinct primes $\equiv 1\pmod{6}$. Remark 7.1 underlines that this can be done for any $s$. Thus every pair $(\deg(f), \deg(g))$ with corresponding $m \in \{1,2,3,4,6\}$ can be represented.

\section{Both $f$ and $g$ have only simple rational roots and \\$(F,G)$ is of the third or fourth kind}
\label{sec34}

If both $f$ and $g$ have simple rational roots, then by symmetry we may assume that $\deg(f)\leq \deg(g)$. Throughout this chapter we shall do so without further mentioning. We show that if in this case the equation $f(x)=g(y)$ has infinitely many rational solutions with a bounded denominator and the corresponding standard pair $(F,G)$ is of the third or fourth kind, then $\deg(F)\leq 2$. Note that $\deg(f)\leq \deg(g)$ implies $\deg(F)\leq\deg(G)$.

\begin{theorem}
\label{impos}
Suppose that $f$ and $g$ have only simple rational roots, and the equation $f(x)=g(y)$ has infinitely many rational solutions with a bounded denominator. If the corresponding standard pair $(F,G)$ is of the third or fourth kind, then $\deg(F)\leq 2$ holds.
\end{theorem}

In the proof we use the following lemmas.

\noindent 
We denote the discriminant of a polynomial $P$ by ${\rm disc}(P)$.

\begin{lemma} [Davenport, Lewis, Schinzel, \cite{dls}, Theorem 1] \label{dals}
Let $F(x) \in \mathbb{Z}[x]$ be of degree $m>1$ and $G(y) \in \mathbb{Z}[x]$ of degree $n>1$.
Let
$$
D(z) = {\rm disc}(F(x) + z),~~E(z) = {\rm disc}(G(y)+z).
$$
Suppose there are at least $[\frac{1}{2} m]$ roots of $D(z) = 0$ for which $E(z) \neq 0$. Then $F(x) - G(y)$ is irreducible over the complex field. Further, the genus of the equation $F(x) - G(y) = 0$ is strictly positive except possibly when $n=2$ or $m=n=3$. Apart from these possible exceptions, the equation has at most a finite number of integral solutions.
\end{lemma}

\begin{lemma}
\label{lemquad}
Let $a,b,c$ be rational numbers such that
$$
3a^2+b^2=c^2.
$$
Then there exist rational numbers $u,v,w$ such that
$$
a=\pm w(2uv),\ \ \ b=\pm w(3u^2-v^2),\ \ \ c=\pm w(3u^2+v^2)
$$
with independent choices of the $\pm$ signs.
\end{lemma}

\begin{proof} If $abc=0$ then we easily see that $a=0$. Then letting $u=0$ and $v=1$, the statement follows.

So we may assume that $a,b,c$ are positive. Then there exists a unique $t\in \Q$ such that $p:=ta$, $q:=tb$, $r:=tc$ are coprime positive integers. Observe that then $p,q,r$ are pairwise coprime, $r>q$ and 
$$
3p^2+q^2=r^2.
$$
By factoring $r^2-q^2$ it follows that the solutions of this equation are given by (cf. A. Desboves \cite{des} and Dickson \cite{dick}, II p. 405 \footnote{Dickson mentions only the second option. By that he misses, for example, the trivial solution $p=q=1, r=2.$})
$$
p=uv,\ \ \ q=\frac{3u^2-v^2}{2},\ \ \ r=\frac{3u^2+v^2}{2}~~{\rm if}~ q-r ~{\rm is ~odd,~ and} 
$$ 
$$
p=2uv,\ \ \ q=3u^2 - v^2,\ \ \ r=3u^2+v^2~~{\rm if}~ q-r ~{\rm is~ even,} 
$$
for some integers $u,v$. Choose $w = t^{-1}$.
\end{proof}

\begin{proof} [Proof of Theorem \ref{impos}]
Suppose that the equation $f(x)=g(y)$ has infinitely many solutions $x,y\in \Q$ with a bounded denominator, and write $(F,G)$ for the corresponding standard pair of the third or fourth kind. Assume that $\deg(F)\geq 3$. Then it follows from Lemma \ref{lemdickson} that $\deg(F),\deg(G)\in \{3,4,6\}$. In view of the gcd-restrictions on standard pairs of the third and fourth kinds it remains to consider $(m,n) := (\deg(F), \deg(G)) = (3,4)$ for the third kind and $=(4,6)$ for the fourth kind.

\vskip.2cm

\noindent
{\bf Standard pairs of the third kind.} We have $(m,n) = (3,4)$. Write
$$
\varphi(x)=p_0(x-p_1)\cdots(x-p_s)
$$
with $p_0\in\Q$, $p_0\neq 0$ and $p_i\in\C$ $(i=1,\dots,s)$. Since the roots of $f(x)=\varphi(F(x))$ are simple and rational, we see that $p_1,\dots,p_s$ are distinct and rational. So we can write 
$$
F(x)-p_i =(x-a_1^{(i)})(x-a_2^{(i)})(x-a_3^{(i)}),
$$
$$
G(y)-p_i =(y-b_1^{(i)})(y-b_2^{(i)})(y-b_3^{(i)})(y-b_4^{(i)}) \ \ (i=1,\dots,s).
$$
Here the $3s$ numbers $a$ form the set of roots of $f$ and are therefore distinct rationals. Similarly the $4s$ numbers $b$ form the set of roots of $g$ and are therefore distinct rationals. We know that the equation
\begin{equation} \label{p_1}
F(x)-p_1=G(y)-p_1
\end{equation}
has infinitely many solutions in rationals $x,y$ with a bounded denominator. Since $F$ is an odd and $G$ is an even polynomial (because they are Dickson polynomials of degree $3$ and $4$, respectively), this implies that the equation (after changing the indexing of the roots if it is necessary)
$$
(x-a_1^{(1)})(x-a_2^{(1)})(x+a_1^{(1)}+a_2^{(1)})=(y^2-(b_1^{(1)})^2)(y^2-(b_2^{(1)})^2)
$$
has infinitely many solutions in rationals $x,y$ with a bounded denominator. Then there exist positive integers $\Delta_1,\Delta_2$  such that, omitting the superscript $(1)$ for simplicity and putting
$$
A_i=\Delta_1a_i\ (i=1,2,3)\ \ \ \text{and}\ \ \ B_j=\Delta_2b_j\ (j=1,2),
$$
the equation
\begin{equation}
\label{maineq34}
U(x):=(x-A_1)(x-A_2)(x+A_1+A_2)=\Delta(y^2-B_1^2)(y^2-B_2^2)=:V(y)
\end{equation}
with $\Delta=\Delta_1^3/\Delta_2^4$ has infinitely many solutions in integers $x,y$.

It follows from Lemma \ref{dals} that, writing
$$
D(z)=\text{disc}(U(x)+z)\ \ \ \text{and}\ \ \ E(z)=\text{disc}(V(y)+z),
$$
every root of $D(z)$ is a root of $E(z)$. A Maple calculation reveals that the roots of 
$D(z)$ are 
\begin{equation} 
-A_1^2A_2 - A_1A_2^2 \pm \frac 29 \sqrt{3(A_1^2+A_1A_2+A_2^2)^3}
\end{equation}
and that the roots of $E(z)$,
\begin{equation}
\label{broots}
-\Delta B_1^2B_2^2,\ \ \ \Delta\left(\frac{B_1^2-B_2^2}{2}\right)^2
\end{equation}
(the latter one being a double root), are rational. So the roots of $D(z)$ have to be rational. Hence, for some $s \in \Q$,
\begin{equation}
\label{quadratic}
3(A_1^2+A_1A_2+A_2^2)=s^2.
\end{equation}
We rewrite \eqref{quadratic} as
$$
3(2A_1+A_2)^2+(3A_2)^2=(2s)^2.
$$
By Lemma \ref{lemquad}, we obtain
$$
2A_1+A_2=\pm w(2uv),\ \ \ 3A_2=\pm w(3u^2-v^2)
$$
with some $u,v,w\in\Q$ and independent choice of the $\pm$ signs. This yields
$$
(A_1,A_2)=w\left(\frac{-3u^2\pm 6uv+v^2}{6}, \pm \frac{3u^2-v^2}{3}\right).
$$
(Here in place of the factor $\pm w$ we can simply write $w$, since $w\in\Q$ is arbitrary.) Therefore the roots of $D(z)$ are given by
$$
\frac{1}{2} w^3u^2(u-v)^2(u+v)^2,\ \ \ -\frac{1}{54} w^3v^2(3u-v)^2(3u+v)^2.
$$
Since, by \eqref{broots}, the products of any two roots \eqref{broots} of $E(z)$ are $\pm$ squares and $2\cdot 54=108$ is not a square in $\Q$, we see that one of the roots of $D(z)$ is zero. Then $E(z)$ has also a root 0. However, then either $B_1B_2=0$ or $B_1=\pm B_2$, which contradicts the distinctness of the roots $B_1,B_2,B_3,B_4$. This contradiction proves that 
\eqref{maineq34} has only finitely many solutions $(x,y) \in \mathbb{Z}^2$, hence \eqref{p_1} and thus also the equation $f(x)=g(y)$ has only finitely many rational solutions with a bounded denominator. So this case cannot occur.

\vskip.2cm

\noindent 
{\bf Standard pairs of the fourth kind.} In this case the only possibility is $(m,n)=(4,6)$, and Theorem \ref{lembt} implies that the standard pair $(F(x),G(y))$ is of the form $(a^{-2}D_4(x,a), -b^{-3}D_6(y,b))$. Further, the equation
\begin{equation} \label{di64}
a^{-2}D_4(x,a)=-b^{-3}D_6(y,b)
\end{equation}
should have infinitely many rational solutions $x,y$ with  bounded denominator. However, by Theorem \ref{lemdickson} we know that here $b$ is of the form $(w_1^2+w_1w_2+w_2^2)/3$ with some $w_1,w_2\in\Q$, in particular, $b>0$. However, since the signs of the leading coefficients of the even degree polynomials in \eqref{di64} are different, this equation can have only finitely many solutions with a bounded denominator. So this case cannot occur either, and the proof of Theorem \ref{impos} is complete.
\end{proof}

\section{A sharpening for $f,g$ with only simple rational roots}
\label{secnew}

We give a refinement of Theorem \ref{caseA} in case both $f$ and $g$ have only simple rational roots. This completes the proof of Theorem \ref{caseA}.
\begin{theorem}
\label{thmnew}
Suppose that $f$ and $g$ have only simple rational roots, and the equation $f(x)=g(y)$ has infinitely many rational solutions with a bounded denominator. Let $k= \deg(f), \ell = \deg(g)$. 
If $k \leq \ell$, then $k \mid 2\ell$, $f$ is a {\rm PTE}$_{m}$-polynomial and  $g$ is a {\rm PTE}$_{\ell m/k }$-polynomial with some $m \in \{1,2\}$. 

Conversely, if $k,\ell$ are positive integers with $k \mid \ell$ and $g$ is a {\rm PTE}$_{\ell /k}$-polynomial of degree $\ell$ with only simple rational roots, then there exists a polynomial $f(x) \in \Q[x]$ with $\deg(f) = k$ and only simple rational roots such that the equation $f(x) = g(y)$ has infinitely many rational solutions with a bounded denominator.
\end{theorem}

\begin{proof}
Suppose that the equation $f(x)=g(y)$ has infinitely many solutions $x,y\in \Q$ with a bounded denominator. Write $(F,G)$ for a corresponding standard pair. Combining Lemma \ref{btcons} and Theorem \ref{impos} we see that $\deg(F)\leq 2$, hence $\deg(f)\mid 2\deg(g)$. Without loss of generality we may assume $F(x)=x^m$ with $m \in \{1,2\}$. If $\varphi(x) = p_0(x-p_1) \cdots (x-p_s)$, then the rationals $p_1, \dots, p_s$ are distinct, $f(x)$ is similar to $p_0(x^m-p_1) \cdots (x^m-p_s)$ and $g(y)$ is similar to $p_0(G(y)-p_1) \cdots (G(y)-p_s)$, which both have simple rational roots. Thus $f$ is a {\rm PTE}$_{m}$-polynomial,  $g$ is a {\rm PTE}$_{\ell m/k }$-polynomial.

Conversely, let $k \mid \ell$ and $g$ be a {\rm PTE}$_{\ell /k}$-polynomial of degree $\ell$ with only simple rational roots. Then $g$ is of the form
\begin{equation} \label{factg}
g(y) = p_0(G(y)-p_1)(G(y)-p_2) \cdots (G(y) - p_k)
\end{equation}
for some $p_0, p_1, \dots, p_k \in \Q$ with $p_1, p_2, \dots, p_k$ distinct.  Write $f(x) = p_0(x-p_1) \cdots (x-p_k)$. Then the equation $f(x) = g(y)$ has solutions $(x,y) = (G(X),X)$ for every $X \in \Z$. 
\end{proof}





\noindent {\bf Remark 9.1.} If in \eqref{factg} $p_i = b_i^2$ for $b_i \in \Q, i=1, \dots, k$, then we may choose $F(x)=x^2, f(x) = (x-b_1)(x+b_1) \cdots (x-b_k)(x+b_k)$.
This is the case $m=2$ in which $f$ is both a PTE$_{1}$-polynomial and a PTE$_{2}$-polynomial.

\vskip.2cm

\noindent {\bf Remark 9.2.} In Sections \ref{sec12} and \ref{sec12fg} we have shown that $\deg(\varphi)$ can be arbitrary, hence $\deg(f), \deg(g)$ can be arbitrarily large. A remaining question is how large $\deg(v)$ in Table 1 can be, if both $f$ and $g$ have only simple rational roots. Without loss of generality we may assume $F(x)=x$ or $F(x)=x^2$. 
Below we treat the cases with the largest degree of $v$. As before we distinguish between the cases $G(y) = \alpha yv(y)^2$ and $G(y) = (\alpha y ^2+ \beta)v(y)^2$ with $\alpha \beta \not= 0$.


\vskip.2cm

\noindent {\bf Example 9.1.} Case $F(x)=x, G(y)=v(y)$. It is known that the sets
$$
T_1 := \{\pm22, \pm61, \pm86, \pm127, \pm140, \pm151\},
$$
$$
T_2 := \{\pm35, \pm47, \pm94, \pm121, \pm146, \pm148\}
$$
form an ideal $\pte_{2,12}$ pair. Let
$$
v(y) =\frac{\prod_{t \in T_1} (y-t) + \prod_{t \in T_2} (y-t)}{2}\ \ \ \text{and}\ \ \ A=\frac{\prod_{t \in T_1} t - \prod_{t \in T_2} t}{2}.
$$ 
Then $$g(y) := \prod_{t \in T_1\cup T_2} (y-t) = (v(y) + A)(v(y) - A) = v(y)^2 - A^2.$$ We define $f(x) = x^2 - A^2$.
Thus $f,g$ both have simple rational roots and the equation $f(x) = g(y)$ has solutions $(x,y) =  (\pm v(X), X)$ for every $X \in \Z$. Here $\varphi(x) = x^2- A^2$.

\vskip.2cm
\noindent
{\bf Example 9.2.} Case $F(x)=x^2, G(y)$ is of the form $\alpha yv(y)^2$. 
\\It is known that the symmetric sets
$$
T_3:=\{-98,-82,-58,-34,13,16,69,75,99\}~{\rm and}~ T_4 := \{t \in T_3 :  -t \}
$$
form an ideal $\pte_{2,9}$ pair. Put $g(y)=\prod_{t\in T_3}(y-t^2)$, $A=\prod_{t \in T_3} t$ and $yT(y)=\prod_{t\in T_3}(y-t)+A$. Then
$$
g(y^2):=\prod_{t\in T_3}(y-t)\cdot \prod_{t\in T_4}(y-t)=(yT(y)-A)(yT(y)+A).
$$
Observe that $yT(y)$ is an odd polynomial, so $T(y)$ is an even polynomial (the coefficients of $y^i$ with $i$ odd are $0$ in $T$). This yields that $T(y)=v(y^2)$ for some $v(y) \in \Q[y]$ and therefore $g(y)=yv(y)^2-A^2$. So letting $f(x)=(x-A)(x+A)$, we see that the equation $f(x)=g(y)$ has solutions $(x,y) = (Xv(X^2), X)$ for every $X \in \Z$.
Here we have $\varphi(x)=x-A^2$, $F(x)=x^2$, $G(y) = yv(y)^2$,  $\deg(g) / \deg(f) = 9/2$.

\vskip.2cm
\noindent
{\bf Example 9.3.} Case $F(x)=x^2, G(y)$ is of the form $(\alpha y^2+\beta)v(y)^2$ with $\alpha \beta \not= 0$. 
An example with $\deg(v)=0$ is given by Example 6.2. 

\section{Equal products from blocks}
\label{seceb}

We give an application of Theorem \ref{caseA} for equal products from blocks of integers of bounded size. By a block we mean a set of consecutive integers.

\begin{theorem} For every positive integer $N$ there exist only finitely many pairs of disjoint blocks $A$ and $B$ of size at most $N$ with the property that for some $k,\ell$ with $1\leq k < \ell$ and $k\nmid 2\ell$, there exist distinct elements $a_1,\dots,a_k\in A$ and distinct elements $b_1,\dots,b_\ell\in B$ such that
\begin{equation}
\label{eq_pr}
a_1\cdots a_k=b_1\cdots b_\ell.
\end{equation}
\end{theorem}

\begin{proof} Suppose the statement of the theorem is false for $N$. We may clearly assume that $k$ and $\ell$ are fixed and that
$$
 a_1<\dots<a_k\ \ \ \text{and}\ \ \ b_1<\dots<b_\ell.
$$
Then we may assume as well that the differences
$$
c_i:=a_i-a_1\ (1<i\leq k)\ \ \ \text{and}\ \ \ d_j:=b_j-b_1\ (1<j\leq \ell)
$$
are fixed. Therefore the equation
$$
f(x) :=x(x+c_1)\dots (x+c_{k-1})=y(y+d_1)\dots (y+d_{\ell-1}) =: g(y)
$$
would have infinitely many solutions in rationals $x,y$ with a bounded denominator. 
By Theorem \ref{thmnew} the corresponding standard pair $(F,G)$ satisfies $\deg(F) \leq 2$. This implies $k\mid 2\ell$, and the statement follows.
\end{proof}

\noindent
{\bf Remark 10.1.}
Example 5.1 provides examples with $k\mid \ell$ such that \eqref{eq_pr} has infinitely many integral solutions. Here $k$ can be arbitrarily large. Examples 5.6 and 9.2 provide examples with $k \mid 2\ell, k \nmid \ell$, and \eqref{eq_pr} has infinitely many integral solutions. 

\section{Conclusions and open problems}

In this paper we have studied equation \eqref{maineq} subject to \eqref{fx} and sometimes \eqref{gy}. In the introduction we formulated the main theorem which demonstrates that the possibilities are restricted, even very restricted if \eqref{gy} holds too. In Section \ref{secbt} we stated the fundamental theorem of Bilu-Tichy that attaches a standard pair of polynomials $(F,G)$ to each equation $f(x)=g(y)$ which has infinitely many solutions. Here and below solutions means solutions in rationals with a bounded denominator. There are five kinds of standard pairs. Lemma \ref{btcons} shows that the equation $f(x)=g(y)$ under \eqref{fx} cannot have infinitely many solutions with the fifth standard pair.
In Sections \ref{sec12} and \ref{sec12fg} we treat the equations leading to standard pairs of the first and second kind. 
Sections \ref{sec23} and \ref{sec34} deal with standard pairs of the third and fourth kind. In Section \ref{sec23} we prove the first statement of Theorem \ref{caseA}. In Section \ref{secnew} we obtain a refinement of the main theorem. Finally we have given an application of the obtained results in Section \ref{seceb}.
\vskip.2cm
Suppose equation \eqref{maineq} for $f(x), g(x) \in \Q[x]$ admits infinitely many integral solutions $(x,y)$ with $f$ subject to \eqref{fx}. 
Put $s = \gcd(\deg(f), \deg(g))$, $m = \deg(f)/s, n= \deg(g)/s.$ At the end of Section 7 we have proved that $m \in \{1,2,3,4,6\}$ or $n \in \{1,2\}$. Moreover we have argued that every pair $(\deg(f), \deg(g))$ with corresponding $m \in \{1,2,3,4,6\}$ can be represented. 
\vskip.1cm
\noindent {\bf Problem 1.} Which other possibilities are there for $m,s$, if $n = 1$ or $2$ for equation \eqref{maineq} under \eqref{fx}?

\vskip.2cm
Now let $f(x), g(x) \in \Q[x]$ both have only simple rational roots and equation \eqref{caseA} have infinitely many integral solutions. We assume $\deg(f) \leq \deg(g)$, hence $m \leq n$. Theorem 1.1 implies $m \in \{1,2\}$. Note that the cases $m=s=1, n$ arbitrary, $m=n$, $s$ arbitrary and $m=1, n=2$, $s$ arbitrary are trivial, the latter in view of
$$(x-b_1^2) \cdots (x-b_s)^2 = (y^2 - b_1^2) \cdots (y^2 - b_s^2), F(x) =x, G(y)=y^2$$
with solutions $(X^2,X)$ for $X \in \Z$. Example 6.1 shows that the cases $m=1$, $n \in \{3,4,6\}$, $s$ arbitrary are possible, cf. Remark 7.2. Example 5.6 deals with the case $m=2$, $n=3$, $s$ arbitrary. Using ideal PTE pairs Example 9.1 can be extended to the cases $m=1$, $n \in \{5, 7, 8, 9, 10, 12\}$, $s=2$ and Example 9.2 to the cases $m=2$, $n \in \{5,7,9\}$, $s=1.$ 

\vskip.1cm
\noindent {\bf Problem 2.} Which other possibilities are there for triples $m, n, s$ for equation \eqref{maineq} under \eqref{fx} and \eqref{gy}?


\section*{Acknowledgements}

Lajos Hajdu would like to express his thanks to the R\'enyi Institute, where he was a visiting professor during this research. The authors thank Szabolcs Tengely for some useful remarks.

\end{document}